\documentclass[10pt]{amsart}
\usepackage{amssymb}
\usepackage{amscd}
\usepackage[all]{xy}

\newcounter{TmpEnumi}

\numberwithin{equation}{section}

\renewcommand{\today}{\number\day\space\ifcase\month\or   January\or February\or
   March\or April\or May\or June\or   July\or August\or September\or
   October\or November\or December\fi\   \number\year}

\theoremstyle{definition}
\newtheorem{thm}{Theorem}[section]
\newtheorem{lem}[thm]{Lemma}
\newtheorem{prp}[thm]{Proposition}
\newtheorem{dfn}[thm]{Definition}
\newtheorem{cor}[thm]{Corollary}
\newtheorem{cnj}[thm]{Conjecture}
\newtheorem{cnv}[thm]{Convention}
\newtheorem{rmk}[thm]{Remark}
\newtheorem{ntn}[thm]{Notation}

\newtheorem{qst}[thm]{Question}

\newcommand{\beq}{\begin{equation}}
\newcommand{\eeq}{\end{equation}}
\newcommand{\beqr}{\begin{eqnarray*}}
\newcommand{\eeqr}{\end{eqnarray*}}
\newcommand{\bal}{\begin{align*}}
\newcommand{\eal}{\end{align*}}
\newcommand{\bei}{\begin{itemize}}
\newcommand{\eei}{\end{itemize}}

\newcommand{\af}{\alpha}
\newcommand{\bt}{\beta}
\newcommand{\gm}{\gamma}
\newcommand{\dt}{\delta}
\newcommand{\ep}{\varepsilon}

\newcommand{\et}{\eta}

\newcommand{\io}{\iota}

\newcommand{\sm}{\sigma}

\newcommand{\ph}{\varphi}

\newcommand{\rh}{\rho}
\newcommand{\om}{\omega}
\newcommand{\ta}{\tau}

\newcommand{\cS}{{\mathcal{S}}}

\newcommand{\cY}{{\mathcal{Y}}}

\newcommand{\Z}{{\mathbb{Z}}}
\newcommand{\R}{{\mathbb{R}}}
\newcommand{\C}{{\mathbb{C}}}
\newcommand{\N}{{\mathbb{Z}}_{> 0}}
\newcommand{\Nz}{{\mathbb{Z}}_{\geq 0}}

\newcommand{\id}{{\operatorname{id}}}

\newcommand{\sint}{{\operatorname{int}}}

\newcommand{\dist}{{\operatorname{dist}}}

\newcommand{\diag}{{\operatorname{diag}}}

\newcommand{\card}{{\operatorname{card}}}
\newcommand{\Aut}{{\operatorname{Aut}}}

\newcommand{\rc}{{\operatorname{rc}}}
\newcommand{\mdim}{{\operatorname{mdim}}}

\newcommand{\Cu}{{\operatorname{Cu}}}

\newcommand{\dirlim}{\varinjlim}

\newcommand{\andeqn}{\,\,\,\,\,\, {\mbox{and}} \,\,\,\,\,\,}

\newcommand{\ts}[1]{{\textstyle{#1}}}

\newcommand{\Wolog}{Without loss of generality}

\newcommand{\ca}{C*-algebra}

\newcommand{\hm}{homomorphism}

\newcommand{\ct}{continuous}
\newcommand{\cfn}{continuous function}

\newcommand{\cms}{compact metric space}
\newcommand{\cpt}{compact Hausdorff}
\newcommand{\chs}{compact Hausdorff space}
\newcommand{\hme}{homeomorphism}
\newcommand{\mh}{minimal homeomorphism}

\newcommand{\rsz}[1]{\raisebox{0ex}[0.8ex][0.8ex]{$#1$}}

\newcommand{\ov}{\overline}
\newcommand{\SM}{\setminus}
\newcommand{\I}{\infty}
\newcommand{\E}{\varnothing}

\title[Radius of comparison and mean dimension]{The C*-algebra
  of a minimal homeomorphism with finite mean dimension has
  finite radius of comparison}

\author{N.~Christopher Phillips}

\date{14~May 2016}   

\address{Department of Mathematics, University  of Oregon,
       Eugene OR 97403-1222, USA.}

\email[]{ncp@darkwing.uoregon.edu}

\thanks{This material is based upon work supported by the
  US National Science Foundation under
  Grants DMS-1101742 and DMS-1501144.
  Some of the work was done during visits by the author to
  the Westf\"{a}lische Wilhelmsuniversit\"{a}t M\"{u}nster
  and the Institut Mittag-Leffler.
  He is grateful to these institutions
  for their hospitality and support.}

\begin{document}

\begin{abstract}
Let $X$ be an infinite compact metric space and let $h \colon X \to X$
be a minimal homeomorphism.
We prove that the radius of comparison
of the crossed product C*-algebra $C^* ( {\mathbb{Z}}, X, h)$
and the mean dimension of~$h$
are related by the inequality
${\operatorname{rc}} (C^* ( {\mathbb{Z}}, X, h))
  \leq 1 + 36 \cdot {\operatorname{mdim}} (h)$.
\end{abstract}

\maketitle

\indent
Let $X$ be an infinite compact metric space and let $h \colon X \to X$
be a homeomorphism.
The mean dimension ${\operatorname{mdim}} (h)$
was introduced in Definition~2.1 of~\cite{LW}
(after being suggested by Gromov).
It behaves well
under the assumption that $h$ does not have ``too many''
periodic points.
Its original application seems to have been to
prove that there are minimal homeomorphisms which do not
equivariantly embed in the shift on $[0, 1]^{\Z}$;
see Corollary~3.4 and Proposition~3.5 of~\cite{LW}.
The radius of comparison ${\operatorname{rc}} (A)$
of a unital \ca~$A$ was introduced in Definition~6.1 of \cite{Tm1}.
An important early application was as the invariant used to
distinguish many mutually nonisomorphic simple unital AH~algebras
with the same Elliott invariant.
(Of course,
these algebras do not have slow dimension growth.)
See Theorem~4.1 of~\cite{Tm7}.

It has been conjectured that if $h$ is minimal
then $\rc (C^* (\Z, X, h)) = \frac{1}{2} \mdim (h)$.
Some partial results are known.
If $\mdim (h) = 0$ then
it follows from Corollary~3.3 of~\cite{EllNiu}
that $\rc (C^* (\Z, X, h)) = 0$.
If $X$ has infinitely many connected components,
then $\rc (C^* (\Z, X, h)) \leq \frac{1}{2} \mdim (h)$
by~\cite{HPT}.
For a class of examples generalizing those of Giol and Kerr~\cite{GK},
including ones with arbitrarily large mean dimension,
we actually have
$\rc (C^* (\Z, X, h)) = \frac{1}{2} \mdim (h)$
(\cite{HPT}).
The difficulty in generalizing
the inequality
$\rc (C^* (\Z, X, h)) \leq \frac{1}{2} \mdim (h)$
from spaces with infinitely many connected components
lies in the technical details
of working with direct limits of recursive subhomogeneous algebras
instead of AH~algebras.
(This is discussed below.)
In this paper,
we use an equivariant embedding
of $(X, h)$ into a shift $K^{\Z}$
(provided by Theorem~5.1 of~\cite{Lnd})
to avoid those difficulties,
and provide a relatively easy proof of the weaker inequality
${\operatorname{rc}} (C^* ( {\mathbb{Z}}, X, h))
  \leq 1 + 36 \cdot {\operatorname{mdim}} (h)$
for an arbitrary \mh{}
of an arbitrary \cms.
In particular,
if ${\operatorname{mdim}} (h)$ is finite
then ${\operatorname{rc}} (C^* ( {\mathbb{Z}}, X, h))$ is finite.

The proof is based on the orbit breaking subalgebras
$C^* (\Z, X, h)_Y$
originally introduced by Putnam in~\cite{Pt1}.
(See Notation~\ref{N_5429_CPN} for the definition.)
In particular,
when $h$ is minimal and $h^k (Y) \cap Y = \E$
for all $k \neq 0$,
then, by Corollary~7.11 of~\cite{PhLg},
$C^* (\Z, X, h)_Y$ is a stably large subalgebra of $C^* (\Z, X, h)$
in the sense of Definition~5.1 of~\cite{PhLg},
and so has the same radius of comparison as $C^* (\Z, X, h)$
(Theorem 6.14 of~\cite{PhLg}).
When $\sint (Y)$ is nonempty,
$C^* (\Z, X, h)_Y$ is a recursive subhomogeneous algebra,
a fact originally observed in unpublished work with Q.~Lin
(see~\cite{LP}).
If $h^k (Y) \cap Y = \E$
for $k \neq 0$,
and $Y_0 \supset Y_1 \supset \cdots$
is a decreasing sequence of compact sets
such that $\bigcap_{n = 0}^{\I} Y_n = Y$,
then $C^* (\Z, X, h)_Y \cong \dirlim_n C^* (\Z, X, h)_{Y_n}$.
If $\sint (Y_n) \neq \E$ for all~$n$,
this expresses $C^* (\Z, X, h)_Y$ as a direct limit of
recursive subhomogeneous algebras.

Niu has given (Theorem~6.2 of~\cite{Niu})
an estimate on $\rc (A)$ when $A$ is
a simple unital AH~algebra with diagonal maps
(and under other minor technical conditions).
It is not clear whether one needs to assume
that the maps are diagonal.
For comparison,
we note that a simple AH~algebra with diagonal maps
necessarily has stable rank one
(Theorem~4.1 of~\cite{EHT}),
whereas without diagonal maps this need not be the case.
It seems rather messy to generalize Niu's result
to simple unital direct limits
of recursive subhomogeneous algebras
with diagonal maps,
although we expect that
the generalization should be true
(with a suitable definition of diagonal maps).
The generalization is expected to require
both that the gluing maps in the
recursive subhomogeneous decompositions be diagonal
and that the maps between
recursive subhomogeneous algebras in the direct system be diagonal.
We need to give a careful description
of a suitable recursive subhomogeneous decomposition
of $C^* (\Z, X, h)_Y$ when $\sint (Y) \neq \E$,
which includes the fact that the gluing maps
are diagonal.
Our proof avoids the need for consideration of
maps between recursive subhomogeneous algebras
except for an extremely special case.

For technical reasons,
we use a nonstandard choice of recursive subhomogeneous decomposition
for $C^* (\Z, X, h)_Y$.
Therefore we give a general scheme.
We assume a weaker condition on $h$ and~$Y$
than minimality of~$h$,
with the hope that it will be useful
for some nonminimal \hme{s},
for example, for ones which have a nontrivial minimal factor system.

In Section~\ref{Sec_Rokh},
we give a careful description
of the recursive subhomogeneous decomposition
of $C^* (\Z, X, h)_Y$
associated to a system of Rokhlin towers for~$Y$
when $Y$ is a sufficiently large subset of~$X$.
When $h$ is minimal and $X$ is infinite,
``sufficiently large'' just means that $\sint (Y) \neq \E$.
Section~\ref{Sec_StrctAY}
contains a preliminary result for this description,
on the existence of \hm{s}
from $C^* (\Z, X, h)_Y$ to algebras of the form
$C (Z, M_N)$ for closed subsets $Z \subset Y$.
All that one needs is $h^N (Z) \subset Y$.
In Section~\ref{Sec_OBSubSf},
we consider a minimal subsystem $(X, h)$ of the shift
on $K^{\Z}$,
and show that,
for suitable choices of~$Y$,
elements of $C^* (\Z, X, h)_Y$ can be approximated by
elements in the homomorphic image of a
recursive subhomogeneous algebra whose base spaces
are closed subspaces of finite products of copies of~$K$.
Section~\ref{Sec_Approx} contains the proof of our main result,
and Section~\ref{Sec_Prob} contains several open problems.

\section{Homomorphisms from $C^* (\Z, X, h)_Y$}\label{Sec_StrctAY}

\indent
Let $X$ be a compact Hausdorff space,
and let $h \colon X \to X$ be a homeomorphism.
Let $Y \subset Z$ be a closed subset,
and let $N \in \N$.
Suppose that $Z \subset Y$ is a nonempty closed subset
which satisfies $h^N (Z) \subset Y$.
We construct and give an explicit formula
for a canonical \hm{}
from the orbit breaking subalgebra $C^* (\Z, X, h)_Y$
(see Notation~\ref{N_5429_CPN})
to $C (Z, M_n)$.
We don't need to assume that $h$ is minimal,
we don't need to assume that $\sint (Y) \neq \E$,
and we don't need to assume that $N$ is the first
return time to $Y$ of any point in~$Z$,
just that $N$ is a return time.

\begin{ntn}\label{N_5429_GpRing}
Let $G$ be a discrete group,
let $A$ be a \ca,
and let $\af \colon G \to \Aut (A)$ be an action of $G$ on~$A$.
We identify $A$ with a subalgebra of $C^*_{\mathrm{r}} (G, A, \af)$
in the standard way.
We let $u_g \in M (C^*_{\mathrm{r}} (G, A, \af))$
be the standard unitary
corresponding to $g \in G$.
We let $A [G]$ denote the dense *-subalgebra
of $C^*_{\mathrm{r}} (G, A, \af)$
consisting of sums $\sum_{g \in S} a_g u_g$
with $S \subset G$ finite and $a_g \in A$ for $g \in S$.
\end{ntn}

\begin{ntn}\label{N_5429_CPN}
Let $X$ be a compact Hausdorff space,
and let $h \colon X \to X$ be a homeomorphism.
The induced automorphism $\af$ of $C (X)$
is given by
$\af (f) = f \circ h^{-1}$ for $f \in C (X)$.
We identify $C (X)$ with its standard image in
$C^* (\Z, X, h)$.
Let $u \in C^* (\Z, X, h)$
be the standard unitary of the crossed product.
Let $E \colon C^* (\Z, X, h) \to C (X)$
be the usual conditional expectation,
so that $u f u^* = f \circ h^{-1}$ for $f \in C (X)$.
Thus
$E \left( \sum_{n = - N}^N f_n u^n \right) = f_0$
for $f_{-N}, \, f_{- N + 1}, \, \ldots, \, f_N \in C (X)$.
For an open subset $U \subset X$
we identify $C_0 (U)$ with
\[
\big\{ f \in C_0 (X) \colon
     {\mbox{$f (x) = 0$ for all $x \in X \setminus U$}} \big\}
 \subset C (X).
\]
For a nonempty closed subset $Y \subset X$,
we let
\[
C^* (\Z, X, h)_Y = C^* \big( C (X), \, C_0 (X \setminus Y) u \big)
  \subset C^* (\Z, X, h)
\]
be the $Y$-orbit breaking subalgebra of
Definition~7.3 of~\cite{PhLg}.
\end{ntn}

We have used a different convention from that used in a number
of other papers,
where one takes
%
\[
C^* (\Z, X, h)_Y = C^* \big( C (X), \, u C_0 (X \setminus Y) \big).
\]
%
Our choice has the advantage that,
when used in connection with Rokhlin towers,
the bases of the towers are subsets of~$Y$
rather than of $h (Y)$.
See the discussion after Definition~7.3 of~\cite{PhLg}.

We recall the identification of $C^* (\Z, X, h)_Y$
from~\cite{PhLg}.

\begin{prp}[Proposition~7.5 of~\cite{PhLg}]\label{P_5429_CharOB}
Let $X$ be a \chs{}, let $h \colon X \to X$ be a \hme,
and let $Y \subset X$ be a nonempty closed subset.
Let $u \in C^* (\Z, X, h)$,
$E \colon C^* (\Z, X, h) \to C (X)$,
and $C^* (\Z, X, h)_Y \subset C^* (\Z, X, h)$
be as in Notation~\ref{N_5429_CPN}.
For $n \in \Z$, set
\[
Y_n = \begin{cases}
   \bigcup_{j = 0}^{n - 1} h^j (Y)    & \hspace{3em} n > 0
        \\
   \varnothing                        & \hspace{3em} n = 0
       \\
   \bigcup_{j = 1}^{- n} h^{-j} (Y)   & \hspace{3em} n < 0.
\end{cases}
\]
Then
\begin{equation}\label{Eq_5429_CharOB}
C^* (\Z, X, h)_Y
 = \big\{ a \in C^* (\Z, X, h) \colon
    {\mbox{$E (a u^{-n}) \in C_0 (X \setminus Y_n)$
          for all $n \in \Z$}} \big\}
\end{equation}
and
\begin{equation}\label{Eq_5429_Dense}
{\overline{C^* (\Z, X, h)_Y \cap C (X) [\Z]}} = C^* (\Z, X, h)_Y.
\end{equation}
\end{prp}

\begin{cnv}\label{Cv_5615_MatUnit}
In this section,
we index the positions in an $n \times n$ matrix
using $\{ 0, 1, \ldots, n - 1 \}$ instead of the usual
$\{ 1, 2, \ldots, n \}$.
Thus,
the standard system of matrix units for~$M_n$
is $(e_{j, k})_{j, k = 0, 1, \ldots, n - 1}$.
\end{cnv}

The following result is related to part of a proof
in~\cite{HPT}.

\begin{prp}\label{L_5429_FromAY}
Let $X$ be a \chs{} and let $h \colon X \to X$ be a \hme.
Let $Z \subset Y \subset X$ be nonempty closed subsets
and let $N \in \N$.
Suppose $h^N (Z) \subset Y$.
Then there exists a unique unital \hm{}
$\gm_{N, Z} \colon C^* (\Z, X, h)_Y \to M_N (C (Z))$
such that for $f \in C (X)$ we have
\[
\gm_{N, Z} (f)
= \diag \big( f |_Z, \, (f \circ h) |_Z,
  \, (f \circ h^{2}) |_Z, \, \ldots,
   \, (f \circ h^{N - 1}) |_Z \big)
\]
and,
taking
\[
w = \left( \begin{matrix}
  0     &  0     & \cdots & \cdots & \cdots &  0        \\
  1     &  0     & \ddots & \ddots & \ddots & \vdots    \\
  0     &  1     & \ddots & \ddots & \ddots & \vdots    \\
 \vdots & \ddots & \ddots & \ddots & \ddots & \vdots    \\
 \vdots & \ddots & \ddots &  1     &  0     &  0        \\
  0     & \cdots & \cdots &  0     &  1     &  0
\end{matrix} \right) \in M_N \subset M_N (C (Z)),
\]
for $f \in C_0 (X \setminus Y)$ we have
\[
\gm_{N, Z} (f u)
 = \diag \big( f |_Z, \, (f \circ h) |_Z,
  \, (f \circ h^{2}) |_Z, \, \ldots,
   \, (f \circ h^{N - 1}) |_Z \big) w.
\]
If $a \in C^* (\Z, X, h)_Y$
then
\begin{equation}\label{Eq_5429_RhNZ}
\gm_{N, Z} (a)
 = \sum_{n = 0}^N \gm_{N, Z} (E (a u^{-n})) w^n
    + \sum_{n = - N}^{- 1} \gm_{N, Z} (E (a u^{-n})) (w^*)^{-n}.
\end{equation}
Moreover, if $a \in C^* (\Z, X, h)_Y$
has the formal series
$\sum_{n = - \I}^{\I} a_n u^n$,
with $a_n \in C (X)$ for $n \in \Z$,
then, following Convention~\ref{Cv_5615_MatUnit},
for $x \in Z$ and
$j, k \in \{ 0, 1, \ldots, N - 1 \}$
we have
\begin{equation}\label{Eq_5615_ByEntry}
\gm_{N, Z} (a) (x)_{j, k} = a_{j - k} (h^j (x)).
\end{equation}
\end{prp}

We don't need to assume that $h$ is minimal,
that $\sint (Y) \neq \E$,
or that $N$ is the first return time of $Z$ to~$Y$.
We don't even need to assume
that the first return time to~$Y$
is constant on~$Z$.

Expressed in terms of matrices,
if $a \in C^* (\Z, X, h)_Y$
has the formal series
$\sum_{n = - \I}^{\I} a_n u^n$
and $x \in Z$,
the formula for $\gm_{N, Y} (a) (x)$ is
\[
\left( \begin{matrix}
 a_0 (x) & a_{-1} (x)  & a_{-2} (x) &
     \cdots  &  a_{- N + 1} (x)        \\
 a_1 ( h (x)) & a_0 ( h (x))  & a_{-1} ( h (x)) &
     \cdots  &  a_{- N + 2} ( h (x))        \\
 a_2 ( h^2 (x)) & a_{1} ( h^2 (x))  & a_{0} ( h^2 (x)) &
     \cdots  &  a_{- N + 3} ( h^2 (x))        \\
 \vdots & \vdots & \vdots & \ddots & \vdots   \\
 a_{N - 1} ( h^{N - 1} (x)) & a_{N - 2} ( h^{N - 1} (x))  &
   a_{N - 3} ( h^{N - 1} (x)) &
     \cdots  &  a_{0} ( h^{N - 1} (x))        \\
\end{matrix} \right).
\]

\begin{proof}[Proof of Proposition~\ref{L_5429_FromAY}]
By the definition of $C^* (\Z, X, h)_Y$,
the \hm{} $\gm_{N, Z}$ is unique if it exists.
Therefore it suffices to show that the
formula~(\ref{Eq_5429_RhNZ})
defines a unital \hm.
This formula obviously defines a \ct{} linear map
$\gm_{N, Z} (a) \colon C^* (\Z, X, h)_Y \to M_N (C (Z))$.
Therefore,
by Proposition~\ref{P_5429_CharOB},
and taking $Y_n$ to be as there,
it suffices to
take $a, b \in C^* (\Z, X, h)$
of the form $a = f u^m$ and $b = g u^n$
with $m, n \in \Z$,
$f \in C_0 (X \SM Y_m)$,
and $g \in C_0 (X \SM Y_n)$,
and show that
$\gm_{N, Z} (a b) = \gm_{N, Z} (a) \gm_{N, Z} (b)$
and $\gm_{N, Z} (a)^* = \gm_{N, Z} (a^*)$.

To simplify the notation,
we define $\gm \colon C (X) \to M_N (C (Z))$
by
\[
\gm (f)
= \diag \big( f |_Z, \, (f \circ h) |_Z,
  \, (f \circ h^{2}) |_Z, \, \ldots,
   \, (f \circ h^{N - 1}) |_Z \big)
\]
for $f \in C (X)$.
(Thus,
$\gm = \gm_{N, Z} |_{C (X)}$.)
We label the standard matrix units in
$M_N \subset M_n (C (Z)) = M_N \otimes C (Z)$
as $e_{j, k}$ for $j, k = 0, 1, \ldots, N - 1$
(following Convention~\ref{Cv_5615_MatUnit}).
We make the following preliminary computations:
for $g \in C (X)$ and $n = 0, 1, \ldots, N - 1$,
we have
\begin{equation}\label{Eq_5429_gsn}
\gm (g) w^n
 = \sum_{k = 0}^{N - n - 1}
  e_{n + k, \, k} \otimes (g \circ h^{n + k}) |_Z
\end{equation}
(matrix entries
$g \circ h^{n} |_Z, \, g \circ h^{n + 1} |_Z, \, \ldots$
in the diagonal starting with the $(0, n)$ entry),
\begin{equation}\label{Eq_5429_sng}
w^n \gm (g)
 = \sum_{k = 0}^{N - n - 1}
  e_{n + k, \, k} \otimes (g \circ h^{k}) |_Z
\end{equation}
(as above but starting with $g |_Z$ rather than
$g \circ h^{n} |_Z$),
\begin{equation}\label{Eq_5429_gstarn}
\gm (g) (w^*)^n
 = \sum_{k = 0}^{N - n - 1}
  e_{k, \, n + k} \otimes (g \circ h^{k}) |_Z
\end{equation}
(matrix entries $g |_Z, \, g \circ h |_Z, \, \ldots$
in the diagonal starting with the $(n, 0)$ entry),
and
\begin{equation}\label{Eq_5429_starng}
(w^*)^n \gm (g)
 = \sum_{k = 0}^{N - n - 1}
  e_{k, \, n + k} \otimes (g \circ h^{n + k}) |_Z
\end{equation}
(as above but starting with $g \circ h^{n} |_Z$ rather than
$g |_Z$).

To check that
$\gm_{N, Z} (a)^* = \gm_{N, Z} (a^*)$,
we need only consider
$a = f u^m$
with $m \geq 0$,
since if $m < 0$
then $f u^m = \big[ ( {\ov{f}} \circ h^{m} ) u^{-m} \big]^*$.
In this case,
if $m \leq N - 1$
we have
\[
\gm_{N, Z} (a)^*
 = [\gm (f) w^n]^*
 = (w^*)^n \gm ( {\ov{f}} )
\]
and
\[
\gm_{N, Z} (a^*)
 = \gm_{N, Z} (u^{-n} {\ov{f}} )
 = \gm_{N, Z} ([{\ov{f}} \circ h^{n}] u^{-n} )
 = \gm ({\ov{f}} \circ h^{n}) (w^*)^n.
\]
To evaluate the first put $g = {\ov{f}}$ in~(\ref{Eq_5429_starng}).
For the second put $g = {\ov{f}} \circ h^{n}$
in~(\ref{Eq_5429_gstarn}),
getting the same element of $M_N (C (Z))$.
If $m \geq N$ then both $\gm_{N, Z} (a)$ and $\gm_{N, Z} (a^*)$
are zero.

We now check multiplicativity.
There are six cases to consider:
\begin{enumerate}
\item\label{Pf_5429_Mult_pp}
$m, n \geq 0$.
\item\label{Pf_5429_Mult_pnp}
$m \geq 0$, $n \leq 0$, and $m + n \geq 0$.
\item\label{Pf_5429_Mult_npp}
$m \leq 0$, $n \geq 0$, and $m + n \geq 0$.
\item\label{Pf_5429_Mult_pnn}
$m \geq 0$, $n \leq 0$, and $m + n \leq 0$.
\item\label{Pf_5429_Mult_npn}
$m \leq 0$, $n \geq 0$, and $m + n \leq 0$.
\item\label{Pf_5429_Mult_nn}
$m, n \leq 0$.
\end{enumerate}
Since we already know that $\gm_{N, Z}$ preserves
adjoints,
by taking adjoints we can obtain
case~(\ref{Pf_5429_Mult_nn})
from case~(\ref{Pf_5429_Mult_pp}),
case~(\ref{Pf_5429_Mult_pnn})
from case~(\ref{Pf_5429_Mult_pnp}),
and case~(\ref{Pf_5429_Mult_npn})
from case~(\ref{Pf_5429_Mult_npp}).
Accordingly,
we only consider the first three cases.

By comparing (\ref{Eq_5429_gsn})
and~(\ref{Eq_5429_sng}),
we get
\begin{equation}\label{Eq_5429_Comm}
w^m \gm (g) = \gm ( g \circ h^{-m}) w^m
\end{equation}
for $m = 0, 1, \ldots, N - 1$ and all $g \in C (X)$.
Accordingly,
if $m, n \geq 0$ and $m + n \leq N - 1$,
we get
\begin{align*}
\gm_{N, Z} (a b)
& = \gm_{N, Z} (f [g \circ h^{-m}] u^{m + n})
\\
& = \gm (f [g \circ h^{-m}]) w^{m + n}
  = \gm (f) w^m \gm (g) w^n
  = \gm_{N, Z} (a) \gm_{N, Z} (b).
\end{align*}
If instead $m + n \geq N$,
one easily checks that
$\gm_{N, Z} (a b) = 0 = \gm_{N, Z} (a) \gm_{N, Z} (b)$.
This is case~(\ref{Pf_5429_Mult_pp}).

We prove case~(\ref{Pf_5429_Mult_pnp}).
It is helpful to rewrite it as
$a = f u^m$ and $b = g u^{-n}$
with $m, n \geq 0$ and $m - n \geq 0$.
First suppose $m \leq N - 1$.
Then
\[
\gm_{N, Z} (a b)
 = \gm_{N, Z} (f [g \circ h^{-m}] u^{m - n})
 = \gm (f [g \circ h^{-m}]) w^{m - n}
\]
and,
using~(\ref{Eq_5429_Comm}) at the second step,
\[
\gm_{N, Z} (a) \gm_{N, Z} (b)
 = \gm (f) w^m \gm (g) (w^*)^n
 = \gm (f [g \circ h^{-m}]) w^m (w^*)^n.
\]

A computation using~(\ref{Eq_5429_gsn}),
and which should be compared with the formula
\[
w^{m - n} - w^m (w^*)^n
 = e_{m - n, \, 0} + e_{m - n + 1, \, 1} + \cdots + e_{m - 1, \, n - 1},
\]
shows that
\begin{equation}\label{Eq_5429_Err}
\gm_{N, Z} (a b) - \gm_{N, Z} (a) \gm_{N, Z} (b)
 = \sum_{k = 0}^{n - 1} e_{m - n + k, \, k}
      \otimes (f \circ h^{m - n + k}) |_Z (g \circ h^{- n + k}) |_Z.
\end{equation}
Since $f u^m \in C^* (\Z, X, h)_Y$,
Proposition~\ref{P_5429_CharOB} implies that
$f$ vanishes on
the sets $Y, \, h (Y), \, \ldots, \, h^{m - 1} (Y)$.
Since $m - n \geq 0$,
it follows that the expression
on the right hand side of~(\ref{Eq_5429_Err})
is zero.

Now suppose $m \geq N$.
Then $\gm_{N, Z} (a) = 0$,
so $\gm_{N, Z} (a) \gm_{N, Z} (b) = 0$.
Also
$a b = f (g \circ h^{- m}) u^{m - n}$.
If $m - n \geq N$ then $\gm_{N, Z} (a b) = 0$,
and we are done.
Otherwise, use Proposition~\ref{P_5429_CharOB}
to see that $f$ vanishes on the sets $Y, h (Y), \ldots, h^{m - 1} (Y)$.
In particular,
$f$ vanishes on the sets
$h^{m - n} (Z), \, h^{m - n + 1} (Z), \, \ldots, \, h^{N - 1} (Z)$,
so $f (g \circ h^{- m})$ also vanishes on these sets.
Thus $\gm_{N, Z} (a b) = 0$ by~(\ref{Eq_5429_gsn}).

Finally,
we prove case~(\ref{Pf_5429_Mult_npp}).
We rewrite this case as
$a = f u^{- m}$ and $b = g u^{n}$
with $m, n \geq 0$ and $n - m \geq 0$.
First suppose $n \leq N - 1$.
Then
\[
\gm_{N, Z} (a b)
 = \gm_{N, Z} (f [g \circ h^{m}] u^{n - m})
 = \gm (f [g \circ h^{m}]) w^{n - m}
\]
and,
using the adjoint of~(\ref{Eq_5429_Comm}) at the second step,
\[
\gm_{N, Z} (a) \gm_{N, Z} (b)
 = \gm (f) (w^*)^m \gm (g) w^n
 = \gm (f [g \circ h^{m}]) (w^*)^m w^n.
\]
A computation using~(\ref{Eq_5429_gsn}),
and which should be compared with the formula
\[
w^{n - m} - (w^*)^m w^n
 = e_{N - n, \, N - m} + e_{N - n + 1, \, N - m + 1}
     + \cdots + e_{N - n + m - 1, \, N - 1},
\]
shows that
\begin{equation}\label{Eq_5429_Err2}
\gm_{N, Z} (a b) - \gm_{N, Z} (a) \gm_{N, Z} (b)
 = \sum_{k = 0}^{m - 1} e_{N - n + k, \, N - m + k}
      \otimes (f \circ h^{N - m + k}) |_Z (g \circ h^{N + k}) |_Z.
\end{equation}
Since $g u^n \in C^* (\Z, X, h)_Y$,
Proposition~\ref{P_5429_CharOB} implies that
$g$ vanishes on
the sets $Y, \, h (Y), \, \ldots, \, h^{n - 1} (Y)$.
Since $h^N (Z) \subset Y$ and $m \leq n$,
it follows that the expression
on the right hand side of~(\ref{Eq_5429_Err2})
is zero.

Now suppose $n \geq N$.
Then $\gm_{N, Z} (b) = 0$,
so $\gm_{N, Z} (a) \gm_{N, Z} (b) = 0$.
Also
$a b = f (g \circ h^{m}) u^{n - m}$.
If $n - m \geq N$ then $\gm_{N, Z} (a b) = 0$,
and we are done.
Otherwise, use Proposition~\ref{P_5429_CharOB}
to see that $g$ vanishes on the sets $Y, h (Y), \ldots, h^{n - 1} (Y)$.
Since $h^N (Z) \subset Y$,
it follows that $g$ vanishes on the sets
$h^{N} (Z), \, h^{N + 1} (Z), \, \ldots, \, h^{N + n - 1} (Z)$.
In particular,
$g$ vanishes on the sets
$h^{n} (Z), \, h^{n + 1} (Z), \, \ldots, \, h^{N + m - 1} (Z)$.
So $f (g \circ h^{m})$ vanishes on the sets
$h^{n - m} (Z), \, h^{n - m + 1} (Z), \, \ldots, \, h^{N - 1} (Z)$,
whence $\gm_{N, Z} (a b) = 0$ by~(\ref{Eq_5429_gsn}).

The formula~(\ref{Eq_5615_ByEntry})
is an immediate consequence of the formula~(\ref{Eq_5429_RhNZ}).
\end{proof}

\section{Rokhlin towers}\label{Sec_Rokh}

\indent
Let $X$ be a compact metric space,
and let $h \colon X \to X$ be a homeomorphism.
Let $Y \subset X$ be a nonempty compact subset
such that
$Y \subset {\ov{\bigcup_{n = 1}^{\I} h^n (Y)}}$.
In this section,
we give a recursive subhomogeneous decomposition
(as in~\cite{Ph6}; see Definition~\ref{D_5619_Rshd} below)
of $C^* (\Z, X, h)_Y$
associated to a system of Rokhlin towers
for which the union of the bases is~$Y$.

If $h$ is minimal and $Y$ is a nonempty compact open
subset of~$X$,
then there is a standard system of Rokhlin towers associated to~$Y$,
in which the bases of the towers form a partition of~$Y$
and the levels of the towers form a partition of~$X$.
When $X$ is connected,
one can't choose such systems.
The right thing to do is to require the levels of the towers
to be closed,
and allow some overlap.
This overlap is the reason that $C^* (\Z, X, h)_Y$
is a recursive subhomogeneous algebra
rather than a homogeneous algebra.
The conventional choice has been to require
that the overlap be as small as possible;
this results in
the Rokhlin system of Lemma~\ref{L_5623_StdTowers} below.
For technical reasons,
the Rokhlin system of Lemma~\ref{L_6224_FullTowers} below
is more suitable for our purposes.
We therefore define a general system of Rokhlin towers.
For possible use elsewhere, the definition
is set up to accommodate more general situations
than \mh{s} and sets with nonempty interior,
although that is the case of interest here.
The purpose of the condition
$Y \subset {\ov{\bigcup_{n = 1}^{\I} h^n (Y)}}$
is to rule out examples in which,
for instance, $h^n (Y)$ is a proper subset of~$Y$
for some $n > 0$.

We continue to follow Convention~\ref{Cv_5615_MatUnit},
indexing matrix entries starting with $0$ instead of~$1$.

We list here,
for convenient reference,
the notation that will be introduced in this section,
together with several previously defined items.
The items in this list
mostly depend on a system $\cY$ of Rokhlin towers.
When dealing with only one such system,
we often suppress it in the notation.
(See Convention~\ref{N_5618_SuppressY}.)

\begin{rmk}\label{R_5618_Summ}
Let $X$ be a compact metric space,
and let $h \colon X \to X$ be a homeomorphism.
Let $Y \subset X$ be a nonempty compact subset
such that
$Y \subset {\ov{\bigcup_{n = 1}^{\I} h^n (Y)}}$.
Let $\cY$ be a Rokhlin system for $(X, h, Y)$.
Then:
\begin{enumerate}
\item\label{R_5618_RetTime}
For $x \in X$, $r_Y$ is the first return time of $x$ to~$Y$
(Definition~\ref{D_5429_RTime}).
\item\label{R_5618_Map}
For $Z \subset Y$ closed
and $N \in \N$ such that $h^N (Z) \subset Y$,
$\gm_{N, Z} \colon C^* (\Z, X, h)_Y \to M_N (C (Z))$
is the \hm{} of Proposition~\ref{L_5429_FromAY}.
\item\label{R_5618_Summ_m}
$m (\cY)$ is the number of towers in the Rokhlin system~$\cY$,
with the count starting at~$0$.
(See Definition~\ref{D_5428_RTY}.)
\item\label{R_5618_Summ_rl}
$r_l (\cY)$ is the height of the $l$-th tower of~$\cY$.
Its levels are indexed from $0$ to $r_l (\cY) - 1$.
(See Definition~\ref{D_5428_RTY}.)
\item\label{R_5618_Summ_T_l}
$T_l (\cY)$ is the base of the $l$-th tower of~$\cY$.
(See Definition~\ref{D_5428_RTY}.)
\item\label{R_5618_Summ_Tl0}
$T_l^0 (\cY)$
is the ``interior'' of the base of the $l$-th tower of~$\cY$,
the points in $T_l (\cY)$ that are not in $T_i (\cY)$ for any $i < l$.
(See Equation~(\ref{Eq_5921_NewStar}) in Definition~\ref{D_5428_RTY}.)
\item\label{R_5618_Summ_DlInt}
$D_l (\cY)$ is the ``boundary'' of the base of the $l$-th tower,
the points in $T_l (\cY)$ that are in $T_i (\cY)$ for some $i < l$.
(See Equation~(\ref{Eq_5921_NewStar}) in Definition~\ref{D_5428_RTY}.)
\item\label{R_5618_Summ_JY}
$J (\cY)$ is the set of labels of towers in the system.
(See Equation~(\ref{Eq_6227_JY}) in Definition~\ref{D_5428_RTY}.)
\item\label{R_5618_Summ_Containing}
$C_l (\cY)$ is the homogeneous algebra
made from the first $l$ towers,
with the count starting at~$0$.
(See Equation~(\ref{Eq_6308_ClCY})
in Definition~\ref{D_5428_RTY_Part2}.)
It contains
the $l$-th stage algebra
(Notation \ref{D_6102_A2}(\ref{D_6102_A2_stage}))
of the recursive subhomogeneous decomposition
for $C^* (\Z, X, h)_Y$
constructed in Theorem~\ref{T_5619_RSHA}.
Moreover, $C (\cY) = C_{m (\cY)} (\cY)$,
the homogeneous algebra which contains
the entire recursive subhomogeneous decomposition
for $C^* (\Z, X, h)_Y$.
(See Equation~(\ref{Eq_6227_StarStar})
in Definition~\ref{D_5428_RTY_Part2}.)
\item\label{R_5618_Summ_gm}
$\gm_{\cY, l}$ is the map from $C^* (\Z, X, h)_Y$
to the homogeneous algebra $C_l (\cY)$.
(See Definition~\ref{D_5428_RTY_Part2}.)
We take $\gm_{\cY} = \gm_{\cY, m (\cY)}$,
the map to the entire homogeneous algebra.
\item\label{R_5618_Summ_cSl}
$\cS_l (\cY)$ is the set of
sequences of towers with index less than~$l$
which points in
$D_l$ could possibly traverse
before getting to the top level of the $l$-th tower.
The criterion is the trivial arithmetic criterion.
(See Definition~\ref{D_5501_S}.)
\item\label{R_5618_Summ_Tlmu}
$T_{l, \mu} (\cY)$ is the set of points in $D_l$
which actually follow the path~$\mu \in \cS_l (\cY)$
through the towers.
(Many of these sets will be empty.
See Definition~\ref{D_5501_S}.)
\item\label{R_5618_Summ_smlmu}
$\bt_{\cY, l, \mu}$ is the contribution made by the points
in $T_{l, \mu} (\cY)$ to the gluing map used to attach
the $l$-th summand in the recursive subhomogeneous decomposition
for $C^* (\Z, X, h)_Y$
constructed in Theorem~\ref{T_5619_RSHA}.
(See Definition~\ref{D_5501_S}.)
\item\label{R_5618_Summ_rshal}
$R_l (\cY)$ is the $l$-th stage
(Notation~\ref{D_6102_A2}(\ref{D_6102_A2_stage}))
in the recursive subhomogeneous decomposition
for $C^* (\Z, X, h)_Y$
constructed in Theorem~\ref{T_5619_RSHA}.
(See Definition~\ref{D_5501_S}.)
\item\label{R_5618_Summ_Xl}
$X_{\cY, l} \subset X$
is the union of the towers labelled $0, 1, \ldots, l$.
(See Convention~\ref{N_5618_SuppressY}.)
\end{enumerate}
\end{rmk}

\begin{dfn}\label{D_5429_RTime}
Let $X$ be a topological space,
and let $h \colon X \to X$ be a homeomorphism.
Let $Y \subset X$ be a subset.
We define the {\emph{first return time}}
$r_Y \colon Y \to \N \cup \{ \I \}$
by
\[
r_Y (y) = \inf \big( \big\{ n \in \N \colon  h^n (y) \in Y \big\} \big)
\]
for $y \in Y$.
\end{dfn}

We have the following standard results.

\begin{lem}\label{L_5429_RTime_lsc}
Adopt the notation of Definition~\ref{D_5429_RTime}.
If $Y$ is closed then $r_Y$ is lower semi\ct.
\end{lem}

\begin{proof}
We must show that for $\af \in \R$,
the set $\big\{ x \in X \colon r_Y (x) > \af \big\}$
is open.
Its complement is
\[
\big\{ x \in X \colon r_Y (x) \leq \af \big\}
 = \bigcup_{n \in [1, \af] \cap \Z} h^{-n} (Y),
\]
which is clearly closed.
\end{proof}

\begin{lem}\label{L_5429_RTime_bdd}
Adopt the notation of Definition~\ref{D_5429_RTime}.
If $h$ is minimal, $X$ is compact Hausdorff,
and $\sint (Y) \neq \E$,
then $\bigcup_{n = 1}^{\infty} h^{-n} (Y) = X$
and $r_Y$ is bounded.
\end{lem}

\begin{proof}
Set $U = \bigcup_{n = 1}^{\infty} h^{-n} (\sint (Y))$,
which is a nonempty open subset of~$X$ such that $U \subset h (U)$.
Then $Z = X \setminus \bigcup_{n = 1}^{\infty} h^{-n} (Y)$
is a closed subset of~$X$ such that $h (Z) \subset Z$,
and $Z \neq X$.
Therefore $Z = \varnothing$.
So $U = X$.
Use compactness of~$X$
to find $N \in \N$ such that $X = \bigcup_{n = 1}^N h^{-n} (Y)$.
Then $r_Y (x) \leq N$ for all $x \in X$.
\end{proof}

We recall the definitions of a
recursive subhomogeneous decomposition
and a recursive subhomogeneous \ca{} from~\cite{Ph6}.

\begin{dfn}[Part of Definition 1.1 of~\cite{Ph6}]\label{D_5619_Rshd}
The class of {\emph{recursive subhomogeneous \ca{s}}}
is the smallest class ${\mathcal{R}}$
of \ca{s} such that:
\begin{enumerate}
\item\label{D_5619_Rshd_Init}
If $T$ is a \cpt{} space and $r \in \N$, then
$C (T, M_r) \in {\mathcal{R}}$.
\item\label{D_5619_Rshd_Ind}
${\mathcal{R}}$ is closed under the following pullback construction:
If $A \in {\mathcal{R}}$, if $T$ is a \cpt{} space,
if $T^{(0)} \subset T$ is closed,
if $\ph \colon A \to C \left( \rsz{ T^{(0)} }, \, M_r \right)$ is any
unital \hm, and if
$\rh \colon C (T, M_r) \to C \left( \rsz{ T^{(0)} }, \, M_r \right)$
is the restriction \hm, then the pullback
\[
A \oplus_{C ( T^{(0)}, \, M_r )} C (T, M_r) =
  \big\{ (a, f) \in A \oplus C (T, M_r) \colon \ph (a) = \rh (f) \big\}
\]
(compare with Definition~2.1 of~\cite{Pd}) is in ${\mathcal{R}}$.
\end{enumerate}

In~(2) the choice $T^{(0)} = \varnothing$ is allowed
(in which case $\ph = 0$ is allowed).
Thus the pullback could be an ordinary direct sum.
\end{dfn}

\begin{ntn}[Part of Definition 1.1 of~\cite{Ph6}]\label{D_6102_A2}
We adopt the following standard notation
for recursive subhomogeneous \ca{s}.
{}From the definition, it is clear that
any recursive subhomogeneous \ca{} can be written in
the form
\[
R = \left[ \cdots \rule{0em}{3ex} \left[ \left[
  C_0 \oplus_{C_1^{(0)}} C_1 \right]
 \oplus_{C_2^{(0)}} C_2 \right] \cdots \right]
            \oplus_{C_m^{(0)}} C_m,
\]
with $C_l = C (T_l, \, M_{r (l)})$ for \cpt{} spaces $T_l$ and positive
integers $r (l)$, with
$C_l^{(0)} = C {\ts{ \left( \rsz{ T_l^{(0)} }, \, M_{r (l)} \right) }}$
for compact
subsets $T_l^{(0)} \subset T_l$ (possibly empty), and where the maps
$C_l \to C_l^{(0)}$ are always the restriction maps.
An expression of this type will be referred to as a
{\emph{decomposition}} of $R$.

Associated with this decomposition are:
\begin{enumerate}
\item\label{D_6102_A2_length}
Its {\emph{length}} $m$.
\item\label{D_6102_A2_stage}
Its {\emph{$l$-th stage algebra}}
\[
R^{(l)} = \left[ \cdots \rule{0em}{3ex} \left[ \left[
  C_0 \oplus_{C_1^{(0)}} C_1 \right]
 \oplus_{C_2^{(0)}} C_2 \right] \cdots \right]
            \oplus_{C_l^{(0)}} C_l,
\]
obtained by using only the first $l + 1$ algebras
$C_0, C_1, \ldots, C_l$.
\item\label{D_6102_A2_base}
Its {\emph{base spaces}} $T_0, T_1, \ldots, T_m$.
\item\label{D_6102_A2_matrix}
Its {\emph{matrix sizes}} $r (0), r (1), \ldots, r (m)$.
\end{enumerate}
\end{ntn}

\begin{dfn}\label{D_5428_RTY}
Let $X$ be a compact metric space,
and let $h \colon X \to X$ be a homeomorphism.
Let $Y \subset X$ be a nonempty compact subset
such that
\begin{equation}\label{Eq_5921_YCond}
Y \subset {\ov{\bigcup_{n = 1}^{\I} h^n (Y)}}.
\end{equation}
Then a {\emph{Rokhlin system}} $\cY$ for $(X, h, Y)$
is a tuple
\[
\cY =
\big( m ({\cY}),
   \, T_0 ({\cY}), \, T_1 ({\cY}), \, \ldots, \, T_{m ( \cY )} ({\cY}),
   \, r_0 ({\cY}), \, r_1 ({\cY}), \, \ldots, \, r_{m ( \cY )} ({\cY})
  \big),
\]
with $m ({\cY}) \in \Nz$,
in which
$T_0 ({\cY}), \, T_1 ({\cY}), \, \ldots, \, T_{m ( \cY )} ({\cY})
 \subset Y$
are nonempty compact subsets,
and in which
$r_0 ({\cY}), \, r_1 ({\cY}), \, \ldots, \, r_{m ( \cY )} ({\cY})
 \in \N$,
satisfying the conditions below.
(The numbers $r_l (\cY)$ are return times,
and there is a conflict with the notation for the
function $r_Y$ of Definition~\ref{D_5429_RTime},
but the meaning should always be clear from the context.)
For convenience in stating these conditions
and for later use,
for $l = 0, 1, \ldots, m ({\cY})$
we define
\begin{equation}\label{Eq_5921_NewStar}
D_l ({\cY}) = T_l ({\cY}) \cap \bigcup_{j = 0}^{l - 1} T_j ({\cY})
\andeqn
T_l^0 ({\cY}) = T_l ({\cY}) \SM D_l ({\cY}).
\end{equation}
The conditions are then as follows:
\begin{enumerate}
\item\label{D_5428_RTY_Union}
$Y = \bigcup_{l = 0}^{m ({\cY})} T_l ({\cY})$.
\item\label{D_5428_RTY_Nondecr}
$r_0 ({\cY}) \leq r_1 ({\cY}) \leq \cdots \leq r_{m ( \cY )} ({\cY})$.
\item\label{D_5428_RTY_Ret}
For $l = 0, 1, \ldots, m ({\cY})$
we have $h^{r_l ({\cY})} (T_l ({\cY})) \subset Y$.
\item\label{D_5428_RTY_ExactRet}
With $r_Y$ as in Definition~\ref{D_5429_RTime},
for $l = 0, 1, \ldots, m ({\cY})$ and $y \in T_l^0 ({\cY})$
we have $r_Y (y) = r_l ({\cY})$.
\item\label{D_5428_RTY_ExistExRt}
With $r_Y$ as in Definition~\ref{D_5429_RTime},
for $l = 0, 1, \ldots, m ({\cY})$, either $T_l ({\cY}) = \E$
and $r_l ({\cY}) \in \{ r_Y (y) \colon y \in Y \}$,
or there exists $y \in T_l ({\cY})$
such that $r_Y (y) = r_l ({\cY})$.
\setcounter{TmpEnumi}{\value{enumi}}
\end{enumerate}
We say that $\cY$ is {\emph{irredundant}} if also:
\begin{enumerate}
\setcounter{enumi}{\value{TmpEnumi}}
\item\label{D_5428_RTY_Irr}
For $l = 0, 1, \ldots, m ({\cY})$,
the set $T_l^0 ({\cY})$
is nonempty and is dense in $T_l ({\cY})$.
\end{enumerate}
The associated {\emph{tower labelling set}} is
\begin{equation}\label{Eq_6227_JY}
J ({\cY})
 = \big\{ (l, j) \in \Nz \times \Nz \colon
  {\mbox{$0 \leq l \leq m ({\cY})$ and $0 \leq j \leq r_l ({\cY}) - 1$}}
 \big\}.
\end{equation}
When $\cY$ is understood,
we may suppress it in the notation.
\end{dfn}

\begin{dfn}\label{D_5428_RTY_Part2}
Let $X$, $h$, $Y$, and~$\cY$
be as in Definition~\ref{D_5428_RTY}.
We define
\begin{equation}\label{Eq_6227_StarStar}
C^* (\Z, X, h)_{\cY} = C^* (\Z, X, h)_Y
\andeqn
C (\cY) = \bigoplus_{l = 0}^{m (\cY)}
    C \big( T_l (\cY), \, M_{r_l ({\cY})} \big).
\end{equation}
For $l = 0, 1, \ldots, m (\cY)$
we define
\begin{equation}\label{Eq_6308_ClCY}
C_l (\cY)
 = \bigoplus_{j = 0}^{l} C \big( T_j (\cY), \, M_{r_j ({\cY})} \big).
\end{equation}
and,
using the notation from
Proposition~\ref{L_5429_FromAY}, we define
\[
\gm_{\cY, l} \colon C^* (\Z, X, h)_{\cY}
 \to C_l (\cY)
\]
by
\[
\gm_{\cY, l} (a)
 = \big( \gm_{r_0 (\cY), \, T_0 (\cY)} (a), \,
         \gm_{r_1 (\cY), \, T_1 (\cY)} (a), \,
         \ldots, \,
         \gm_{r_{l} (\cY), \, T_{l} (\cY)} (a) \big)
\]
for $a \in C^* (\Z, X, h)_{\cY}$.
We further define
$\gm_{\cY} = \gm_{\cY, m (\cY)} \colon C^* (\Z, X, h)_{\cY} \to C (\cY)$
and for $l = 0, 1, \ldots, m (\cY)$ we let
$\pi_{\cY, l} \colon C (\cY)
 \to \bigoplus_{j = 0}^{l} C \big( T_j (\cY), \, M_{r_j ({\cY})} \big)$
be the projection to the first $l + 1$ coordinates,
so that $\gm_{\cY, l} = \pi_{\cY, l} \circ \gm_{\cY}$.
\end{dfn}

The sets
$T_0 ({\cY}), \, T_1 ({\cY}), \, \ldots, \, T_{m ( \cY )} ({\cY})$
are the bases of the Rokhlin towers associated with~$\cY$.
For fixed~$l$,
the sets
$T_l ({\cY}), \, h (T_l ({\cY})),
 \, \ldots, \, h^{r_l ({\cY}) - 1} (T_l ({\cY}))$
form the Rokhlin tower with base $T_l ({\cY})$.
We need to use closed sets,
so that objects like $C (\cY)$ as defined above make sense,
but if $Y$ is not both closed and open it is then
not possible to choose the sets in the towers to be disjoint.
The set $D_l (\cY)$ is the ``boundary'' of $T_l (\cY)$
in $\bigcup_{j = 0}^{l} T_j (\cY)$,
that is,
the part of $T_l (\cY)$ which is used when gluing
it onto $\bigcup_{j = 0}^{l - 1} T_j (\cY)$.
The set $T_l^0 ({\cY})$ is the ``interior'' of $T_l ({\cY})$.

The hypotheses on~$Y$ in Definition~\ref{D_5428_RTY}
are not sufficient to ensure that Rokhlin systems exist.
They do exist in the cases relevant here.
The standard example is given in
Lemma~\ref{L_5623_StdTowers}.
The one we actually use is given in
Lemma~\ref{L_6224_FullTowers}.
The difference is in the definition of $T_l ({\cY})$;
the sets $T_l^0 ({\cY})$
are the same.
The proofs are easy, and are omitted.

\begin{lem}\label{L_5623_StdTowers}
Adopt the notation of Definition~\ref{D_5429_RTime}.
Assume that $h$ is minimal and $\sint (Y) \neq \E$.
Then the following definitions give a Rokhlin system~$\cY$
for $(X, h, Y)$
as in Definition~\ref{D_5428_RTY}:
\begin{enumerate}
\item\label{L_5623_StdTowers_m}
$m ({\cY}) = \card ( \{ r_Y (y) \colon y \in Y \} ) - 1$.
\item\label{L_5623_StdTowers_rl}
The range of $r_Y$
is
$\big\{ r_0 ({\cY}), \, r_1 ({\cY}),
 \, \ldots, \, r_{m ( \cY )} ({\cY}) \big\}$,
labelled in increasing order:
$r_0 ({\cY}) < r_1 ({\cY}) < \cdots < r_{m ( \cY )}$.
\item\label{L_5623_StdTowers_TlY}
$T_l ({\cY})$ is the closure of the set
$\{ y \in Y \colon r_Y (y) = r_l ({\cY}) \}$
for $l = 0, 1, \ldots, m ({\cY})$.
\setcounter{TmpEnumi}{\value{enumi}}
\end{enumerate}
Moreover, this Rokhlin system is irredundant,
and we have:
\begin{enumerate}
\setcounter{enumi}{\value{TmpEnumi}}
\item\label{L_5623_StdTowers_37}
$\bigcup_{n \in \Z} h^n (Y) = X$.
\item\label{L_5623_StdTowers_T0}
$T_l^0 ({\cY}) = \{ y \in Y \colon r_Y (y) = r_l ({\cY}) \}$
for $l = 0, 1, \ldots, m ({\cY})$.
\end{enumerate}
\end{lem}

\begin{lem}\label{L_6224_FullTowers}
Adopt the notation of Definition~\ref{D_5429_RTime}.
Assume that $h$ is minimal and $\sint (Y) \neq \E$.
Then the following definitions give a Rokhlin system~$\cY$
for $(X, h, Y)$
as in Definition~\ref{D_5428_RTY}:
\begin{enumerate}
\item\label{L_6224_FullTowers_m}
$m ({\cY}) = \card ( \{ r_Y (y) \colon y \in Y \} ) - 1$.
\item\label{L_6224_FullTowers_rl}
The range of $r_Y$
is
$\big\{ r_0 ({\cY}), \, r_1 ({\cY}),
 \, \ldots, \, r_{m ( \cY )} ({\cY}) \big\}$,
labelled in increasing order:
$r_0 ({\cY}) < r_1 ({\cY}) < \cdots < r_{m ( \cY )}$.
\item\label{L_6224_FullTowers_TlY}
$T_l ({\cY})
 = \big\{ y \in Y \colon h^{r_l ({\cY})} (y) \in Y \big\}$
for $l = 0, 1, \ldots, m ({\cY})$.
\setcounter{TmpEnumi}{\value{enumi}}
\end{enumerate}
Moreover, we have:
\begin{enumerate}
\setcounter{enumi}{\value{TmpEnumi}}
\item\label{L_6224_FullTowers_37}
$\bigcup_{n \in \Z} h^n (Y) = X$.
\item\label{L_6224_FullTowers_T0}
$T_l^0 ({\cY}) = \{ y \in Y \colon r_Y (y) = r_l ({\cY}) \}$
for $l = 0, 1, \ldots, m ({\cY})$.
\end{enumerate}
\end{lem}

\begin{lem}\label{L_5430_Basics}
Let $X$, $h$, $Y$, and~$\cY$
be as in Definition~\ref{D_5428_RTY},
and follow the notation there.
(See Remark~\ref{R_5618_Summ}.)
Then:
\begin{enumerate}
\item\label{L_5430_Basics_CvY}
$Y = \coprod_{l = 0}^{m ({\cY})} T_l^0 ({\cY})$.
\item\label{L_5430_Basics_R}
For $l = 0, 1, \ldots, m (\cY)$
and $y \in T_l (\cY)$,
we have $r_Y (y) \leq r_l (\cY)$.
\item\label{L_5430_Basics_N}
With $N = r_{m ( \cY )} (\cY)$,
the set $\bigcup_{n \in \Z} h^n (Y)$
is compact and equal to $\bigcup_{n = 0}^{N - 1} h^n (Y)$.
\item\label{L_5430_Basics_CovX}
$\bigcup_{n \in \Z} h^n (Y)
  = \coprod_{l = 0}^{m ({\cY})} \coprod_{j = 0}^{r_l ({\cY}) - 1}
   h^j (T_l^0 ({\cY}))
  = \coprod_{(l, j) \in J (\cY)} h^j (T_l^0 ({\cY}))$.
\item\label{L_5725_NReturn_One_2}
$Y = \coprod_{l = 0}^{m ({\cY})} h^{r_l ({\cY})} (T_l^0 ({\cY}))$.
\item\label{L_5725_NReturn_Two}
For $n \in \Nz$ we have
\[
\bigcup_{j = 0}^{n - 1} h^j (Y)
 = \coprod_{l = 0}^{m ({\cY})}
    \coprod_{j = 0}^{\min (n, \, r_l ({\cY})) - 1} h^j (T_l^0 ({\cY})).
\]
\item\label{L_5725_NReturn_Two_2}
For $n \in \N$ we have
\[
\bigcup_{j = 1}^{n} h^{-j} (Y)
 = \coprod_{l = 0}^{m ({\cY})}
    \coprod_{j = 1}^{\min (n, \, r_l ({\cY}))}
   h^{r_l ({\cY}) - j} (T_l^0 ({\cY})).
\]
\end{enumerate}
\end{lem}

\begin{proof}
Part~(\ref{L_5430_Basics_CvY}) is immediate from
$\bigcup_{l = 0}^{m ({\cY})} T_l ({\cY}) = Y$
and (\ref{Eq_5921_NewStar})
in Definition~\ref{D_5428_RTY}.
Part~(\ref{L_5430_Basics_R})
is immediate from
Definition~\ref{D_5428_RTY}(\ref{D_5428_RTY_Ret}).

Define
\[
Z = \bigcup_{(l, j) \in J (\cY)} h^j (T_l^0 ({\cY})).
\]
We claim that
\begin{equation}\label{Eq_5921_InY}
\bigcup_{n = 0}^{\I} h^n (Y) = Z.
\end{equation}
That $Z \subset \bigcup_{n = 0}^{\I} h^n (Y)$ is clear.
For the reverse inclusion,
let $n \in \Nz$ and let $x \in h^n (Y)$.
Set $y = h^{-n} (x) \in Y$.
Let $n_0 \in \{ 0, 1, \ldots, n \}$
be the largest integer in this set such that $h^{n_0} (y) \in Y$.
By part~(\ref{L_5430_Basics_CvY}),
there is $l \in \{ 0, 1, \ldots, m (\cY) \}$
such that $h^{n_0} (y) \in T_l^0 (\cY)$.
By the definition of $r_Y$ we have
\[
n - n_0 < r_Y ( h^{n_0} (y) ) = r_l (\cY).
\]
Therefore
\[
x = h^{n - n_0} ( h^{n_0} (y) )
  \in h^{n - n_0} ( T_l^0 (\cY) )
  \subset Z.
\]
This completes the proof of~(\ref{Eq_5921_InY}).

With $N = r_{m ( \cY )} (\cY)$,
we now have
\begin{equation}\label{Eq_5921_Chain}
Z
= \bigcup_{(l, j) \in J (\cY)} h^j (T_l^0 ({\cY}))
\subset \bigcup_{(l, j) \in J (\cY)} h^j (T_l({\cY}))
\subset \bigcup_{n = 0}^{N - 1} h^n (Y)
\subset \bigcup_{n = 0}^{\I} h^n (Y)
= Z.
\end{equation}
So
$\bigcup_{n = 0}^{\I} h^n (Y) = \bigcup_{n = 0}^{N - 1} h^n (Y)$.
Applying $h$,
we get
$\bigcup_{n = 1}^{\I} h^n (Y) = \bigcup_{n = 1}^{N} h^n (Y)$.
Therefore $\bigcup_{n = 1}^{\I} h^n (Y)$ is closed,
and it follows
from~(\ref{Eq_5921_YCond}) in Definition~\ref{D_5428_RTY}
that $Y \subset \bigcup_{n = 1}^{\I} h^n (Y)$.
{}From~(\ref{Eq_5921_InY}), we now get
\begin{equation}\label{Eq_5921_YhZ}
Y \subset h (Z).
\end{equation}

We next claim that the sets
$h^j (T_l^0 ({\cY}))$,
for $(l, j) \in J (\cY)$,
are disjoint.
So let $(l_1, j_1), \, (l_2, j_2) \in J (\cY)$
and assume that
\[
h^{j_1} \big( T_{l_1}^0 ({\cY}) \big)
     \cap h^{j_2} \big( T_{l_2}^0 ({\cY}) \big)
 \neq \E.
\]
\Wolog{} $j_1 \leq j_2$.
Then
\[
T_{l_1}^0 ({\cY}) \cap h^{j_2 - j_1} \big( T_{l_2}^0 ({\cY}) \big)
 \neq \E
\andeqn
0 \leq j_2 - j_1 \leq r_{l_2} (\cY) - 1.
\]
If $j_2 - j_1 \neq 0$
then $h^{j_2 - j_1} (T_{l_2}^0 ({\cY})) \cap Y = \E$
by Definition~\ref{D_5428_RTY}(\ref{D_5428_RTY_ExactRet})
while $T_{l_1}^0 ({\cY}) \subset Y$.
This contradiction shows that $j_1 = j_2$.
Now $T_{l_1}^0 ({\cY}) \cap T_{l_2}^0 ({\cY}) \neq \E$
so $l_1 = l_2$
by part~(\ref{L_5430_Basics_CvY}).
The claim is proved.

At this point,
we know that the unions in the second and third expressions
in part~(\ref{L_5430_Basics_CovX})
are disjoint and are equal.

We can now write
\[
h (Z)
 = \coprod_{l = 0}^{m ({\cY})} \coprod_{j = 1}^{r_l ({\cY})}
   h^j (T_l^0 ({\cY})),
\]
and moreover
\[
Y \cap \coprod_{l = 0}^{m ({\cY})} \coprod_{j = 1}^{r_l ({\cY}) - 1}
   h^j (T_l^0 ({\cY}))
 = \E.
\]
Since $Y \subset h (Z)$,
we get
\[
Y \subset \coprod_{l = 0}^{m ({\cY})} h^{r_l ({\cY})} (T_l^0 ({\cY})).
\]
Since also
$h^{r_l ({\cY})} (T_l^0 ({\cY})) \subset Y$
for $l = 0, 1, \ldots, m (\cY)$
(by Definition~\ref{D_5428_RTY}(\ref{D_5428_RTY_Ret})),
we deduce that
\[
Y = \coprod_{l = 0}^{m ({\cY})} h^{r_l ({\cY})} (T_l^0 ({\cY})).
\]
This is part~(\ref{L_5725_NReturn_One_2}).

We next claim that $\bigcup_{n = 1}^{\I} h^{-n} (Y) \subset Z$.
So suppose $n \in \N$ and $x \in h^{-n} (Y)$.
Thus $h^n (x) \in Y$.
If $x \in Y$,
we are done since $Y \subset Z$.
Otherwise,
let $n_0$ be the least integer in $\{ 1, 2, \ldots, n \}$
such that $h^{n_0} (x) \in Y$.
By part~(\ref{L_5725_NReturn_One_2}),
there is $l \in \{ 0, 1, \ldots, m (\cY) \}$
such that $h^{n_0} (x) \in h^{r_l (\cY)} (T_l^0 (\cY) )$.
Then $h^{n_0 - r_l (\cY)} (x) \in T_l^0 (\cY) \subset Y$,
so $n_0 - r_l (\cY) < 0$.
Since $n_0 \neq 0$,
we get
$0 \leq r_l (\cY) - n_0 \leq r_l (\cY) - 1$.
Then
\[
x \in h^{r_l (\cY) - n_0} (T_l^0 (\cY)) \subset Z.
\]
The claim is proved.

Combining this claim with~(\ref{Eq_5921_InY})
finishes the proof of part~(\ref{L_5430_Basics_N}),
and combining it with~(\ref{Eq_5921_Chain})
finishes the proof of part~(\ref{L_5430_Basics_CovX}).

It remains to prove parts (\ref{L_5725_NReturn_Two})
and~(\ref{L_5725_NReturn_Two_2}).
We prove both parts
by induction on~$n$.
For the first part,
the case $n = 0$ is part~(\ref{L_5430_Basics_CvY}).
Assuming the statement is true for~$n$,
we clearly have
\[
\bigcup_{j = 0}^{n - 1} h^j (Y)
 \subset \coprod_{l = 0}^{m ({\cY})}
   \coprod_{j = 0}^{\min (n + 1, \, r_l ({\cY})) - 1}
         h^j (T_l^0 ({\cY}))
 \subset \bigcup_{j = 0}^{n} h^j (Y).
\]
Moreover,
\begin{equation}\label{Eq_5411_Diff}
\bigcup_{j = 0}^{n} h^j (Y)
 \setminus \coprod_{l = 0}^{m ({\cY})}
   \coprod_{j = 0}^{\min (n + 1, \, r_l ({\cY})) - 1}
          h^j (T_l^0 ({\cY}))
\end{equation}
is contained in the union of
those sets $h^n (T_l^0 ({\cY}))$ for which $r_l (\cY) = n$.
Since $h^{r_l (\cY)} (T_l^0 ({\cY})) \subset Y$,
it follows that the set~(\ref{Eq_5411_Diff})
is empty.
This is~(\ref{L_5725_NReturn_Two}).

The proof of~(\ref{L_5725_NReturn_Two_2})
is essentially the same,
except that we start from
part~(\ref{L_5725_NReturn_One_2})
instead of from part~(\ref{L_5430_Basics_CvY}).
\end{proof}

\begin{lem}\label{L_5725_NReturn}
Let the notation be as in Definition~\ref{D_5428_RTY}
and Definition~\ref{D_5428_RTY_Part2}
(see Remark~\ref{R_5618_Summ}),
and in addition assume that $\bigcup_{n \in \Z} h^n (Y) = X$.
Then:
\begin{enumerate}
\item\label{L_5725_NReturn_One}
$X \setminus Y
 = \coprod_{l = 0}^{m ({\cY})} \coprod_{j = 1}^{r_l ({\cY}) - 1}
   h^j (T_l^0 ({\cY}))$.
\item\label{L_5430_BasicsrhInj}
The map $\gm_{\cY}$
of Definition~\ref{D_5428_RTY_Part2} is injective.
\end{enumerate}
\end{lem}

\begin{proof}
Part~(\ref{L_5725_NReturn_One}) is immediate from
Lemma \ref{L_5623_StdTowers}(\ref{L_5430_Basics_CovX})
and Lemma \ref{L_5623_StdTowers}(\ref{L_5430_Basics_CvY}).

We prove part~(\ref{L_5430_BasicsrhInj}).
Let $a \in C^* (\Z, X, h)_{\cY}$ be nonzero.
Let $\sum_{m = - \I}^{\I} a_m u^m$ be the corresponding
formal series.
Then there is $n \in \N$ such that $a_n \neq 0$.

First assume $n \geq 0$.
By Proposition~\ref{P_5429_CharOB},
there is $x \in X \SM \bigcup_{j = 0}^{n - 1} h^j (Y)$
such that $a_n (x) \neq 0$.
Combining
Lemma \ref{L_5623_StdTowers}(\ref{L_5430_Basics_CovX})
and Lemma \ref{L_5623_StdTowers}(\ref{L_5725_NReturn_Two}),
we find $l, j \in \Nz$ such that
\[
0 \leq l \leq m (\cY),
\,\,\,\,\,\,
\min (n, r_l ({\cY})) \leq j \leq r_l (\cY) - 1,
\andeqn
x \in h^j (T_l (\cY)).
\]
We must clearly have $n \leq j \leq r_l (\cY) - 1$.
Applying~(\ref{Eq_5615_ByEntry}) in Proposition~\ref{L_5429_FromAY},
we get
\[
\gm_{r_l (\cY), \, T_l (\cY)} (a) ( h^{-j} (x))_{j, j - n}
 = a_n (x)
 \neq 0.
\]
Referring to Definition~\ref{D_5428_RTY},
we get $\gm_{\cY} (a) \neq 0$.

Now suppose $n < 0$.
Set $b = a^*$,
and let $\sum_{m = - \I}^{\I} b_m u^m$ be the corresponding
formal series.
Then $b_{-n} \neq 0$,
so the case already done implies that $\gm_{\cY} (b) \neq 0$.
Thus $\gm_{\cY} (a) = \gm_{\cY} (b)^* \neq 0$.
\end{proof}

\begin{proof}
The proof of~(\ref{Eq_5725_NrUniq_Ret}) for $y \in T_l^0 ({\cY})$
is immediate
from Definition \ref{D_5428_RTY}(\ref{D_5428_RTY_ExactRet}).
Uniqueness of $r_l ({\cY})$ then follows
from the assumption that $T_l^0 ({\cY}) \neq \E$.
\end{proof}

\begin{dfn}\label{D_5501_S}
Let the notation be as in Definition~\ref{D_5428_RTY}
and Definition~\ref{D_5428_RTY_Part2}
(see Remark~\ref{R_5618_Summ}),
but abbreviate $m (\cY) = m$, $r_l (\cY) = r_l$,
and $T_l (\cY) = T_l$.
For $l = 0, 1, \ldots, m$,
we define the set $\cS_l (\cY)$
of {\emph{$l$-admissible sequences}}
to be the set of all finite sequences
$\mu = (\mu (1), \, \mu (2), \, \ldots, \mu (t))$
in $\{ 0, 1, \ldots, l - 1 \}$
such that $\sum_{s = 1}^t r_{\mu (s)} = r_l$.
For such a sequence~$\mu$,
we define the length of~$\mu$ to be~$t$,
and
we set
\begin{align*}
T_{l, \mu} (\cY)
& = \big\{ x \in T_{l} \colon
   x \in T_{\mu (1)}, \,
   h^{r_{\mu (1)}} (x) \in T_{\mu (2)}, \,
\\
& \hspace*{3em} {\mbox{}}
   h^{r_{\mu (1)} + r_{\mu (2)}} (x) \in T_{\mu (3)}, \,
   \ldots, \,
   h^{r_{\mu (1)} + r_{\mu (2)} + \cdots + r_{\mu (t - 1)}}
      \in T_{\mu (t)}
    \big\},
\end{align*}
(a closed set---this is stated in
Lemma \ref{L_6103_Combined}(\ref{L_6103_Combined_TlmClosed})),
and further define
\[
\bt_{\cY, l, \mu} \colon
 \bigoplus_{i = 0}^{l - 1}
    C ( T_i, \, M_{r_i} )
\to C ( T_{l, \mu} (\cY), \, M_{r_l} )
\]
by
\begin{align*}
& \bt_{\cY, l, \mu} (b_0, b_1, \ldots, b_{l - 1} ) (x)
\\
& \hspace*{3em} {\mbox{}}
= \diag \big( b_{\mu (1)} (x),
       \, (b_{\mu (2)} \circ h^{r_{\mu (1)}}) (x),
\\
& \hspace*{7em} {\mbox{}}
  \, (b_{\mu (3)} \circ h^{r_{\mu (1)} + r_{\mu (2)}}) (x), \, \ldots,
   \, (b_{\mu (t)}
    \circ h^{r_{\mu (1)} + r_{\mu (2)} + \cdots + r_{\mu (t - 1)}}) (x)
   \big)
\end{align*}
for
\[
b_0 \in C ( T_0, \, M_{r_0} ), \,
b_1 \in C ( T_1, \, M_{r_1} ), \,
\ldots, \,
b_{l - 1} \in C ( T_{l - 1}, \, M_{r_{l - 1}} ),
\andeqn
x \in T_{l, \mu} (\cY).
\]

Inductively define subalgebras
$R_l (\cY) \subset \bigoplus_{i = 0}^{l} C ( T_i, \, M_{r_i} )$
by $R_0 (\cY) = C ( T_0, \, M_{r_0} )$
and
\begin{align}\label{Eq_5618_RSHD}
R_l (\cY)
& =
\Big\{ (b_0, b_1, \ldots, b_{l})
    \in R_{l - 1} (\cY) \oplus C ( T_l, \, M_{r_l} )
   \colon
\\
& \hspace*{1.5em} {\mbox{}}
 {\mbox{$b_l (x) = \bt_{\cY, l, \mu} (b_0, b_1, \ldots, b_{l - 1}) (x)$
   for all $\mu \in \cS_l (\cY)$ and all $x \in T_{l, \mu} (\cY)$}}
    \Big\}
\notag
\end{align}
for $l = 1, \ldots, m$.
\end{dfn}

The main goal of the rest of this section
is to prove Theorem~\ref{T_5619_RSHA},
which states that
$R_{m (\cY)} (\cY)$ is equal to the range of~$\gm_{\cY}$,
and that
the description in~(\ref{Eq_5618_RSHD})
is a recursive subhomogeneous decomposition.


\begin{cnv}\label{N_5618_SuppressY}
When $\cY$ is unambiguous,
we suppress it in the notation.
Thus,
in Definition~\ref{D_5428_RTY},
Definition~\ref{D_5428_RTY_Part2},
and Definition~\ref{D_5501_S},
we abbreviate
\[
m (\cY) = m,
\,\,\,\,\,\,
r_l (\cY) = r_l,
\,\,\,\,\,\,
T_l (\cY) = T_l,
\,\,\,\,\,\,
T_l^0 (\cY) = T_l^0,
\,\,\,\,\,\,
D_l (\cY) = D_l,
\]
\[
\gm_{\cY, l} = \gm_l,
\,\,\,\,\,\,
\cS_l (\cY) = \cS_l,
\,\,\,\,\,\,
T_{l, \mu} (\cY) = T_{l, \mu},
\,\,\,\,\,\,
\bt_{\cY, l, \mu} = \bt_{l, \mu},
\andeqn
R_l (\cY) = R_l.
\]
We also define,
still suppressing~$\cY$,
\[
X_l = \bigcup_{i = 0}^{l} \bigcup_{j = 0}^{r_i - 1} h^j (T_i)
\]
for $l = 0, 1, \ldots, m$.
\end{cnv}

\begin{lem}\label{L_6103_Combined}
Let the notation be as in Convention~\ref{N_5618_SuppressY}.
Let $l \in \{ 0, 1, \ldots, m \}$.
Then:
\begin{enumerate}
\item\label{L_6103_Combined_TlmClosed}
$T_{l, \mu}$ is closed for all $\mu \in \cS_l$.
\item\label{L_6103_Combined_Unital}
For all $\mu \in \cS_l$,
$\bt_{l, \mu}$ is a unital \hm.
\item\label{L_6103_Combined_DlUnion_P}
For every $x \in D_l$ there are $t \in \N$
and $\mu \in \cS_l$ of length~$t$ such that
$r_{\mu (1)} + r_{\mu (2)} + \cdots + r_{\mu (t)} = r_l$
and for
$s = 1, 2, \ldots, t$
we have
\[
h^{r_{\mu (1)} + r_{\mu (2)} + \cdots + r_{\mu (s)}} (x)
 \in T^0_{\mu (s)}.
\]
\item\label{L_6103_Combined_DlUnion}
$D_l = \bigcup_{\mu \in \cS_l} T_{l, \mu}$.
\item\label{L_6103_Combined_UseInt}
We have
\[
X_l = \coprod_{i = 0}^{l} \coprod_{j = 0}^{r_i - 1} h^j (T_i^0).
\]
\item\label{L_6103_Combined_j1j2InXl1}
Let $j_1, j_2 \in \{ 0, 1, \ldots, r_l - 1 \}$ be distinct.
Then $h^{j_1} (T_l) \cap h^{j_2} (T_l) \subset X_{l - 1}$.
\item\label{L_6103_Combined_j1j2InXl1_New}
Let $j_1, j_2 \in \{ 0, 1, \ldots, r_l - 1 \}$ be distinct.
Then $h^{j_1} (T_l) \cap h^{j_2} (T_l^0) = \E$.
\item\label{L_6103_Combined_hjxTl}
Let $j \in \{ 0, 1, \ldots, r_l - 1 \}$,
and suppose $x \in T_l$ satisfies $h^j (x) \in X_{l - 1}$.
Then $x \in D_l$.
\item\label{L_6103_Combined_TlXl_1}
$D_l = T_l \cap X_{l - 1}$.
\end{enumerate}
\end{lem}

\begin{proof}
Part~(\ref{L_6103_Combined_TlmClosed}) is immediate.

We prove~(\ref{L_6103_Combined_Unital}).
Since $T_{l, \mu}$ is closed by part~(\ref{L_6103_Combined_TlmClosed}),
the only thing to check is that $\bt_{l, \mu}$ is unital.
This follows from the relation $\sum_{s = 1}^t r_{\mu (s)} = r_l$.

We prove~(\ref{L_6103_Combined_DlUnion_P}).
We inductively construct a sequence $\mu$
in $\{ 0, 1, \ldots, l - 1 \}$ of length~$t$,
such that
$h^{r_{\mu (1)} + r_{\mu (2)} + \cdots + r_{\mu (s)}} (x) \in Y$
for $s = 1, 2, \ldots, t$, as follows.
By definition,
$x \in \bigcup_{i = 0}^{l - 1} T_i$,
so by Lemma~\ref{L_5430_Basics}(\ref{L_5430_Basics_CvY})
there is $i \in \{ 0, 1, \ldots, l - 1 \}$
such that $x \in T_i^0$.
Define $\mu (1) = i$.
Then $h^{r_{\mu (1)}} (x) \in Y$ by definition.

Given $\mu (1), \, \mu (2), \, \ldots, \, \mu (s)$,
set $p = r_{\mu (1)} + r_{\mu (2)} + \cdots + r_{\mu (s)}$.
If $p = r_l$,
stop, and set $t = s$
and $\mu = (\mu (1), \, \mu (2), \, \ldots, \, \mu (t))$.
Since $h^{r_l} (x) \in Y$,
we must otherwise have $p < r_l$ and $h^p (x) \in Y$.
By Lemma \ref{L_5430_Basics}(\ref{L_5430_Basics_CvY}),
there is $i \in \{ 0, 1, \ldots, m \}$
such that $h^p (x) \in T_i^0$.
We have
\[
r_i = r_Y (h^p (x))
    \leq r_l - p
    < r_l,
\]
so $i < l$ by Definition \ref{D_5428_RTY}(\ref{D_5428_RTY_Nondecr}).
Define $\mu (s + 1) = i$.
Then
$h^{r_{\mu (1)} + r_{\mu (2)} + \cdots + r_{\mu (s + 1)}} (x) \in Y$
by definition.
This completes the induction step.
Part~(\ref{L_6103_Combined_DlUnion_P}) follows.

We prove~(\ref{L_6103_Combined_DlUnion}).
For $\mu \in \cS_l$,
we have $T_{l, \mu} \subset D_l$ by definition.
The reverse inclusion follows from~(\ref{L_6103_Combined_DlUnion_P}).

We prove~(\ref{L_6103_Combined_UseInt}).
The sets on the right in the desired conclusion are disjoint
by Lemma \ref{L_5430_Basics}(\ref{L_5430_Basics_CovX}),
so we need to show that
$X_l = \bigcup_{i = 0}^{l} \bigcup_{j = 0}^{r_i - 1} h^j (T_i^0)$.
It suffices to show that
\begin{equation}\label{Eq_6308_Star}
\bigcup_{i = 0}^{l} \bigcup_{j = 0}^{r_i - 1} h^j (D_i)
 \subset \bigcup_{i = 0}^{l} \bigcup_{j = 0}^{r_i - 1} h^j (T_i^0).
\end{equation}
For $k = 0, 1, \ldots, l$
define
\[
Z_k = \bigcup_{i = 0}^{k} \bigcup_{j = 0}^{r_i - 1} h^j (T_i^0).
\]
We claim that for $j = 0, 1, \ldots, r_k - 1$,
we have $h^j (D_k) \subset Z_k$.

Let $x \in D_k$ and let $j \in \{ 0, 1, \ldots, r_k - 1 \}$.
Choose $\mu \in \cS_k$
as in~(\ref{L_6103_Combined_DlUnion_P}).
Let $t$ be the length of~$\mu$.
Then choose $s \in \{ 0, 1, \ldots, t - 1 \}$ such that,
with $p = r_{\mu (1)} + r_{\mu (2)} + \cdots + r_{\mu (s)}$,
we have $p \leq j < p + r_{\mu (s + 1)}$.
Now $h^p (x) \in T^0_{\mu (s + 1)}$ and
$j - p \leq r_{\mu (s + 1)} - 1$,
so
\[
h^j (x) = h^{j - p} (h^p (x))
 \in h^{j - p} \big( T^0_{\mu (s + 1)} \big)
 \subset Z_k.
\]
The claim is proved.

Now
\[
\bigcup_{i = 0}^{l} \bigcup_{j = 0}^{r_i - 1} h^j (D_i)
 \subset \bigcup_{i = 0}^{l} Z_i
 = Z_l,
\]
which is~(\ref{Eq_6308_Star}).
This completes the proof of~(\ref{L_6103_Combined_UseInt}).

We prove~(\ref{L_6103_Combined_j1j2InXl1}).
\Wolog{} $j_1 < j_2$.
Let $x \in h^{j_1} (T_l) \cap h^{j_2} (T_l)$.
Set $y = h^{-j_2} (x) \in T_l$.
Then $h^{j_2 - j_1} (y) = h^{-j_1} (x) \in T_l$.
Thus $r_Y (y) \leq j_2 - j_1 < r_l$.
Therefore $y \in D_l$.
Part~(\ref{L_6103_Combined_DlUnion}) provides $\mu \in \cS_l$
such that $y \in T_{l, \mu}$.
Let $t$ be the length of~$\mu$.
Choose $s \in \{ 0, 1, \ldots, t - 1 \}$ such that,
with
\[
p = r_{\mu (1)} + r_{\mu (2)} + \cdots + r_{\mu (s)},
\]
we have $p \leq j_2 < p + r_{\mu (s + 1)}$.
Then
\[
x = h^{j_2} (y) \in h^{j_2 - p} (T_{\mu (s + 1)}) \subset X_{l - 1}.
\]
This completes the proof of~(\ref{L_6103_Combined_j1j2InXl1}).

Part~(\ref{L_6103_Combined_j1j2InXl1_New})
follows from part~(\ref{L_6103_Combined_j1j2InXl1}),
since applying part~(\ref{L_6103_Combined_UseInt})
for $l$ and also for $l - 1$ in place of~$l$
shows that $h^{j_2} (T_l^0) \cap X_{l - 1} = \E$.

For~(\ref{L_6103_Combined_hjxTl}),
if $x \in T_l^0$
then disjointness in~(\ref{L_6103_Combined_UseInt})
implies that $h^j (x) \not\in X_{l - 1}$
for $j = 0, 1, \ldots, r_l - 1$.

We prove~(\ref{L_6103_Combined_TlXl_1}).
It is clear that $D_l \subset T_l \cap X_{l - 1}$.
For the reverse,
take $j = 0$ in part~(\ref{L_6103_Combined_hjxTl}).
\end{proof}

\begin{lem}\label{L_5615_RanGmContained}
Let the notation be as in Convention~\ref{N_5618_SuppressY},
Definition~\ref{D_5428_RTY_Part2},
and Proposition~\ref{L_5429_FromAY}.
Assume that
\[
l \in \{ 0, 1, \ldots, m \},
\,\,\,\,\,\,
\mu \in \cS_l,
\,\,\,\,\,\,
x \in T_{l, \mu},
\andeqn
a \in C^* (\Z, X, h)_{\cY}.
\]
Then $\gm_{r_l, T_l} (a) (x) = \bt_{l, \mu} ( \gm_{l - 1} (a)) (x)$.
\end{lem}

\begin{proof}
We follow Convention~\ref{Cv_5615_MatUnit}.
Let $j, k \in \{ 0, 1, \ldots, r_l - 1 \}$.
We show that
$\gm_{r_l, T_l} (a) (x)_{j, k}
 = \bt_{l, \mu} ( \gm_{l - 1} (a)) (x)_{j, k}$.
Since we are dealing with *-\hm{s},
we need only consider the case $j \geq k$.

Let $\sum_{n = \I}^{\I} a_n u^n$
be the formal series for~$a$.
The statement~(\ref{Eq_5615_ByEntry})
in Proposition~\ref{L_5429_FromAY}
gives $\gm_{r_l, T_l} (a) (x)_{j, k} = a_{j - k} (h^j (x))$.

Let $t$ be the length of~$\mu$.
By definition,
there exists $s \in \{ 0, 1, \ldots, t - 1 \}$
such that,
setting
\[
p = r_{\mu (1)} + r_{\mu (2)} + \cdots + r_{\mu (s)},
\]
we have
\[
p \leq j < p + r_{\mu (s + 1)}.
\]

There are two cases.
First suppose that $k < p$.
The matrix
$\bt_{l, \mu} ( \gm_{l - 1} (a)) (x)$
is block diagonal
with blocks of sizes
$r_{\mu (1)}, \, r_{\mu (2)}, \, \ldots, \, r_{\mu (t - 1)}$,
and the $(j, k)$ position is in none of these blocks,
so $\bt_{l, \mu} ( \gm_{l - 1} (a)) (x)_{j, k} = 0$.
Also,
$a_{j - k} (h^j (x)) = a_{j - k} (h^{j - p} (h^p (x)))$.
Since $h^p (x) \in Y$
and $j - p < j - k$,
it follows from Proposition~\ref{P_5429_CharOB}
that $a_{j - k} (h^{j - p} (h^p (x))) = 0$.
Thus
$\gm_{r_l, T_l} (a) (x)_{j, k}
 = \bt_{l, \mu} ( \gm_{l - 1} (a)) (x)_{j, k}$.

Otherwise,
we have $p \leq k \leq j$.
For $y \in T_{\mu (s + 1)}$,
we then have
\[
\gm_{r_{\mu (s + 1)}, T_{\mu (s + 1)}} (y)_{j - p, \, k - p}
 = a_{j - k} (h^{j - p} (y))
\]
by the statement~(\ref{Eq_5615_ByEntry})
in Proposition~\ref{L_5429_FromAY}.
Using this equation with $y = h^p (x)$
at the second step,
and the definitions of $\bt_{l, \mu}$ and $\gm_{l - 1}$
at the first step,
we get
\begin{align*}
\bt_{l, \mu} ( \gm_{l - 1} (a)) (x)_{j, k}
& = \gm_{r_{\mu (s + 1)}, T_{\mu (s + 1)}} ( h^p (x))_{j - p, \, k - p}
\\
& = a_{j - k} (h^{j - p} (h^p (x)))
  = \gm_{r_l, T_l} (a) (x)_{j, k},
\end{align*}
as desired.
This completes the proof.
\end{proof}

\begin{cor}\label{C_5618_RanGmContained_Alt}
Let the notation be as in Convention~\ref{N_5618_SuppressY}.
Let $l \in \{ 0, 1, \ldots, m \}$
and let
\[
b = (b_0, b_1, \ldots, b_l)
 \in \bigoplus_{j = 0}^{l} C \big( T_j, M_{r_j} \big)
\]
be in the range of $\gm_l$.
Then for all $\mu \in \cS_l$ we have
\[
b_l |_{T_{l, \mu}} = \bt_{l, \mu} (b_0, b_1, \ldots, b_{l - 1}).
\]
\end{cor}

\begin{proof}
This follows from
Lemma~\ref{L_5615_RanGmContained}
by induction on~$l$.
\end{proof}

\begin{cor}\label{C_5618_RanGmContained}
Let the notation be as in Convention~\ref{N_5618_SuppressY}.
For $l = 0, 1, \ldots, m$,
the range of $\gm_l$ is contained in~$R_l$.
\end{cor}

\begin{proof}
This is immediate from
Corollary~\ref{C_5618_RanGmContained_Alt}
and (\ref{Eq_5618_RSHD}) in Definition~\ref{D_5501_S}.
\end{proof}

\begin{lem}\label{L_5618_Surj}
Let the notation be as in Convention~\ref{N_5618_SuppressY}.
Let $l \in \{ 0, 1, \ldots, m \}$,
and let $b \in R_l$.
Then there exists $a \in C^* (\Z, X, h)_{\cY}$
such that $\gm_l (a) = b$.
\end{lem}

\begin{proof}
The proof is by induction on~$l$.
We follow Convention~\ref{Cv_5615_MatUnit}.

Suppose first that $l = 0$.
We have $R_0 = C (T_0, M_{r_0})$ and $\gm_0 = \gm_{r_0, T_0}$.
Let $b \in R_0$,
and write $b = \sum_{j, k = 0}^{r_0 - 1} e_{j, k} \otimes b_{j, k}$
with $b_{j, k} \in C (T_0)$ for $j, k = 0, 1, \ldots, r_0 - 1$.
We prove that for $n = 0, 1, \ldots, r_0 - 1$
the element
$c = \sum_{j = n}^{r_0 - 1} e_{j, j - n} \otimes b_{j, j - n}$
is in the range of $\gm_{r_0, T_0}$.
Since $\gm_{r_0, T_0}$ is a *-\hm,
knowing this for all $b \in R_0$
and all $n \in \{ 0, 1, \ldots, r_0 - 1 \}$ will imply the result.

Since $r_Y (y) \geq r_0$ for all $y \in Y$,
the sets
\[
Y, \, h (Y), \, \ldots, \, h^{n - 1} (Y), \,
  h^n (T_0), \, h^{n + 1} (T_0), \, \ldots, \, h^{r_0 - 1} (T_0)
\]
are disjoint closed subsets of~$X$.
Therefore the Tietze Extension Theorem provides $f \in C (X)$
such that $f |_{h^{j} (Y)} = 0$
for $j = 0, 1, \ldots, n - 1$
and $f |_{h^{j} (T_0)} = b_{j, j - n} \circ h^{-j}$
for $j = n, \, n + 1, \, \ldots, \, r_0 - 1$.
It follows from Proposition~\ref{P_5429_CharOB}
that $f u^n \in C^* (\Z, X, h)_{\cY}$
and from the formula~(\ref{Eq_5615_ByEntry})
in Proposition~\ref{L_5429_FromAY}
that $\gm_{r_0, T_0} (f u^n) = c$.

Now assume the result is known for $l - 1$;
we prove it for~$l$.
Let $b = (b_0, b_1, \ldots, b_l) \in R_l$.
We first consider the case
$b_0 = b_1 = \cdots = b_{l - 1} = 0$.
Write
\[
b_l = \sum_{j, k = 0}^{r_l - 1} e_{j, k} \otimes (b_l)_{j, k}
\]
with $(b_l)_{j, k} \in C (T_l)$ for $j, k = 0, 1, \ldots, r_l - 1$.
We prove that for $n = 0, 1, \ldots, r_l - 1$
and taking
\[
c = \sum_{j = n}^{r_l - 1} e_{j, j - n} \otimes (b_l)_{j, j - n},
\]
the element $(0, 0, \ldots, 0, c)$
is in the range of $\gm_{l}$.
Since $\gm_{l}$ is a *-\hm,
knowing this for all $b$
such that $(0, 0, \ldots, 0, b) \in R_l$
and all $n \in \{ 0, 1, \ldots, r_l - 1 \}$ will prove the result
in this case.

We first claim that there exists $f \in C (X)$
such that:
\begin{enumerate}
\item\label{L_5618_Surj_Xl_1}
$f |_{X_{l - 1}} = 0$.
\item\label{L_5618_Surj_Smallj}
$f |_{h^{j} (Y)} = 0$
for $j = 0, 1, \ldots, n - 1$.
\item\label{L_5618_Surj_Largej}
$f |_{h^{j} (T_l)} = (b_l)_{j, j - n} \circ h^{-j}$
for $j = n, \, n + 1, \, \ldots, \, r_l - 1$.
\setcounter{TmpEnumi}{\value{enumi}}
\end{enumerate}
The subsets of~$X$ involved in (\ref{L_5618_Surj_Xl_1}),
(\ref{L_5618_Surj_Smallj}), and~(\ref{L_5618_Surj_Largej})
are all closed,
so it suffices to show that the formulas agree on
all the intersections of these sets,
and then apply the Tietze Extension Theorem.
Thus, we prove:
\begin{enumerate}
\setcounter{enumi}{\value{TmpEnumi}}
\item\label{L_5618_Surj_BothBig}
If $j_1, j_2 \in \{ n, \, n + 1, \, \ldots, \, r_l - 1 \}$
and $x \in h^{j_1} (T_l) \cap h^{j_2} (T_l)$,
then
\[
(b_l)_{j_1, j_1 - n} (h^{- j_1} (x))
 = (b_l)_{j_2, j_2 - n} (h^{- j_2} (x)).
\]
\item\label{L_5618_Surj_SmallBig}
If $j_1 \in \{ 0, 1, \ldots, n - 1 \}$,
$j_2 \in \{ n, \, n + 1, \, \ldots, \, r_l - 1 \}$,
and $x \in h^{j_1} (Y) \cap h^{j_2} (T_l)$,
then $(b_l)_{j_2, j_2 - n} (h^{- j_2} (x)) = 0$.
\item\label{L_5618_Surj_BigXl1}
If $j \in \{ 0, 1, \ldots, r_l - 1 \}$
and $x \in h^j (T_l) \cap X_{l - 1}$,
then $(b_l)_{j, j - n} (h^{- j} (x)) = 0$.
\end{enumerate}
(We really only need~(\ref{L_5618_Surj_BigXl1})
for $j \in \{ n, \, n + 1, \, \ldots, \, r_l - 1 \}$,
but it is convenient to take it as stated.
For all other possible intersections of two of the sets in
(\ref{L_5618_Surj_Xl_1}),
(\ref{L_5618_Surj_Smallj}), and (\ref{L_5618_Surj_Largej}),
the two formulas automatically agree because both give zero.)

We first prove~(\ref{L_5618_Surj_BigXl1}).
Set $y = h^{- j} (x)$.
Then $y \in T_l$ and $h^j (y) \in X_{l - 1}$,
so $y \in D_l$ by
Lemma \ref{L_6103_Combined}(\ref{L_6103_Combined_hjxTl}).
Lemma \ref{L_6103_Combined}(\ref{L_6103_Combined_DlUnion})
provides $\mu \in \cS_l$ such that $y \in T_{l, \mu}$.
By hypothesis (see Definition~\ref{D_5501_S}),
we have $b_l (y) = \bt_{l, \mu} (b_0, b_1, \ldots, b_{l - 1}) (y)$,
which is zero since $b_0 = b_1 = \cdots = b_{l - 1} = 0$.
So $(b_l)_{j, j - n} (h^{- j} (x)) = (b_l)_{j, j - n} (y) = 0$,
as desired.

For~(\ref{L_5618_Surj_BothBig}),
we can obviously assume $j_1 \neq j_2$.
Then $x \in X_{l - 1}$
by Lemma \ref{L_6103_Combined}(\ref{L_6103_Combined_j1j2InXl1}).
Two applications of~(\ref{L_5618_Surj_BigXl1})
now show that
\[
(b_l)_{j_1, j_1 - n} (h^{- j_1} (x))
 = 0
 = (b_l)_{j_2, j_2 - n} (h^{- j_2} (x)).
\]

For~(\ref{L_5618_Surj_SmallBig}),
we can assume that $j_1$ is the least number in
$\{ 0, 1, \ldots, n - 1 \}$ such that $x \in h^{j_1} (Y)$.
Set $y = h^{- j_1} (x)$.
Then $y \in Y$ but $h^j (y) \not\in Y$
for $j = 1, 2, \ldots, j_1$.
By Lemma~\ref{L_5430_Basics}(\ref{L_5430_Basics_CvY}),
there is $i \in \{ 0, 1, \ldots, m \}$
such that $y \in T_i^0$.
Then $j_1 \leq r_i - 1$.
Suppose first $i > l$.
We have
\[
h^{j_1} (y) = x \in h^{j_2} (T_l) \subset X_l \subset X_{i - 1}.
\]
Since $y \in T_i$ and $j_1 \leq n \leq r_l - 1  \leq r_i - 1$,
Lemma \ref{L_6103_Combined}(\ref{L_6103_Combined_hjxTl})
implies $y \in D_i$.
Since $D_i \cap T_i^0 = \E$,
this is a contradiction.
Next, suppose $i = l$.
Then Lemma \ref{L_6103_Combined}(\ref{L_6103_Combined_j1j2InXl1_New})
implies $j_1 = j_2$,
contradicting $j_1 \leq n - 1 < n \leq j_2$.
Thus we must have $i < l$.
Since $j_1 \leq r_i - 1$,
we have $x = h^{j_1} (y) \in X_{l - 1}$.
Thus $x \in h^{j_2} (T_l) \cap X_{l - 1}$,
so $(b_l)_{j_2, j_2 - n} (h^{- j_2} (x)) = 0$
by~(\ref{L_5618_Surj_BigXl1}).
The claim is proved.

Let $f$ be as in the claim.
It follows from Proposition~\ref{P_5429_CharOB}
and part~(\ref{L_5618_Surj_Smallj}) of the claim
that $f u^n \in C^* (\Z, X, h)_{\cY}$.
Combining the formula~(\ref{Eq_5615_ByEntry})
in Proposition~\ref{L_5429_FromAY}
with part~(\ref{L_5618_Surj_Xl_1}) of the claim
gives $\gm_{r_i, T_i} (f u^n) = 0$
for $i = 0, 1, \ldots, l - 1$,
and combining it with part~(\ref{L_5618_Surj_Largej}) of the claim
gives $\gm_{r_l, T_l} (f u^n) = c$.
The completes the proof of the special case
$b_0 = b_1 = \cdots = b_{l - 1} = 0$.

We now drop the assumption that
$b_0 = b_1 = \cdots = b_{l - 1} = 0$.
It is clear that $(b_0, b_1, \ldots, b_{l - 1}) \in R_{l - 1}$.
By the induction hypothesis,
there is $a_0 \in C^* (\Z, X, h)_{\cY}$
such that $\gm_{l - 1} (a_0) = (b_0, b_1, \ldots, b_{l - 1})$.
Then $\gm_l (a_0) \in R_l$ by Corollary~\ref{C_5618_RanGmContained}.
Therefore
\[
\big( 0, \, 0, \, \ldots, \, 0, \, b_l - \gm_{r_l, T_l} (a_0) \big)
 = b - \gm_l (a_0)
 \in R_l.
\]
By the case already done,
there is $a_1 \in C^* (\Z, X, h)_{\cY}$
such that
$\gm_l (a_1) = b - \gm_l (a_0)$.
Then $a = a_0 + a_1 \in C^* (\Z, X, h)_{\cY}$
and satisfies $\gm_l (a) = b$.
\end{proof}

The following lemma is now easy.
Waiting until now to prove it
avoids very complicated notation.

\begin{lem}\label{L_5619_WellDf}
Let the notation be as in Convention~\ref{N_5618_SuppressY}.
Let $l \in \{ 1, 2, \ldots, m \}$,
let $b \in R_{l - 1}$,
let $\mu, \nu \in \cS_l$,
and let $x \in T_{l, \mu} \cap T_{l, \nu}$.
Then $\bt_{l, \mu} (b) (x) = \bt_{l, \nu} (b) (x)$.
\end{lem}

\begin{proof}
Use Lemma~\ref{L_5618_Surj}
to choose $a \in C^* (\Z, X, h)_{\cY}$
such that $\gm_{l - 1} (a) = b$.
Apply Lemma~\ref{L_5615_RanGmContained} twice to get
\[
\bt_{l, \mu} (b) (x)
 = \bt_{l, \mu} ( \gm_{l - 1} (a)) (x)
 = \gm_{r_l, T_l} (a) (x)
 = \bt_{l, \nu} ( \gm_{l - 1} (a)) (x)
 = \bt_{l, \nu} (b) (x).
\]
This completes the proof.
\end{proof}

\begin{lem}\label{L_5619_DfnSm}
Let the notation be as in Convention~\ref{N_5618_SuppressY}.
Let $l \in \{ 1, 2, \ldots, m \}$.
Then there is a unique unital \hm{}
$\bt_{\cY, l} \colon R_{l - 1} \to C (D_l, M_{r_l})$
(abbreviated to $\bt_l$ when $\cY$ is understood)
such that for all $\mu \in \cS_l$,
all $x \in T_{l, \mu}$,
and all $b \in R_{l - 1}$
we have $\bt_{\cY, l} (b) (x) = \bt_{l, \mu} (b) (x)$.
\end{lem}

\begin{proof}
That such a function exists and is unique
follows from
Lemma \ref{L_6103_Combined}(\ref{L_6103_Combined_TlmClosed}),
Lemma \ref{L_6103_Combined}(\ref{L_6103_Combined_DlUnion}),
and Lemma~\ref{L_5619_WellDf}.
Lemma \ref{L_6103_Combined}(\ref{L_6103_Combined_Unital})
implies that $\bt_l$ is a \hm{} and is unital.
\end{proof}

The following theorem is the analog of a result in~\cite{HPT}
without the assumption in~\cite{HPT}
that $X$ has enough compact open subsets.

\begin{thm}\label{T_5619_RSHA}
Let the notation be as in Convention~\ref{N_5618_SuppressY}.
Then $R_{m}$ has a recursive subhomogeneous decomposition
as in Definition~\ref{D_5619_Rshd},
with,
following Notation~\ref{D_6102_A2},
base spaces
$T_0, T_1, \ldots, T_m$
and matrix sizes
$r_0, r_1, \ldots, r_m$,
and given by
\begin{align*}
R_m
& = \bigg[ \cdots \left[ \left[
  C (T_0, M_{r_0}) \oplus_{C (D_1, M_{r_1})} C (T_1, M_{r_1}) \right]
 \oplus_{C (D_2, M_{r_2})} C (T_2, M_{r_2}) \right]
\\
& \hspace*{3em} {\mbox{}}
 \oplus_{C (D_3, M_{r_3})} \cdots \bigg]
            \oplus_{C (D_m, M_{r_m})} C (T_m, M_{r_m}),
\end{align*}
in which,
for $l = 0, 1, \ldots, m$,
the $l$-th stage algebra is $R_l$,
the map $C (T_l, M_{r_l}) \to C (D_l, M_{r_l})$
is the restriction map~$\rh_l$,
and the map
$R_{l - 1} \to C (D_l, M_{r_l})$
is the map $\bt_l$ of Lemma~\ref{L_5619_DfnSm}.
Moreover,
$\gm_m \colon C^* (\Z, X, h)_{\cY} \to R_m$
is surjective,
and is an isomorphism if $\bigcup_{n \in \Z} h^n (Y) = X$.
\end{thm}

\begin{proof}
For the first statement,
we only need to check that,
for $l = 1, 2, \ldots, m$,
the map
\[
(b_0, b_1, \ldots, b_l)
 \mapsto \big( (b_0, b_1, \ldots, b_{l - 1}), \, b_l \big)
\]
induces an isomorphism
\[
\ph_l \colon R_{l}
 \to R_{l - 1}
    \oplus_{C (D_l, M_{r_l}), \, \bt_l, \, \rh_l} C (T_l, M_{r_l}).
\]

We show that $\ph_l$ is well defined,
that is,
that if $b = (b_0, b_1, \ldots, b_l) \in R_l$,
then $\ph_l (b)$ really is in the pullback.
Lemma~\ref{L_5618_Surj}
provides $a \in C^* (\Z, X, h)_{\cY}$
such that $\gm_l (a) = b$.
Corollary~\ref{C_5618_RanGmContained_Alt}
and Lemma~\ref{L_5619_DfnSm}
imply that
$\rh_l (b_l) = \bt_{l} (b_0, b_1, \ldots, b_{l - 1})$,
as desired.

It is obvious that $\ph_l$ is injective.

For surjectivity,
let $\big( (b_0, b_1, \ldots, b_{l - 1}), \, b_l \big)$
be in the pullback.
Then $(b_0, b_1, \ldots, b_l) \in R_l$
by the definition of~$\bt_l$ in Lemma~\ref{L_5619_DfnSm}
and the definition of~$R_l$
in Definition~\ref{D_5501_S}.

In the last statement,
the range of $\gm_m$ is contained in $R_m$
by Corollary~\ref{C_5618_RanGmContained},
and contains $R_m$ by Lemma~\ref{L_5618_Surj}.
The map $\gm_m$ is the same as the map~$\gm_{\cY}$
(see Definition~\ref{D_5428_RTY_Part2}),
so the injectivity statement follows from
Lemma~\ref{L_5725_NReturn}(\ref{L_5430_BasicsrhInj}).
\end{proof}

\section{Orbit breaking subalgebras for subshifts}\label{Sec_OBSubSf}

\indent
Let $K$ be a compact metric space,
and let $h \colon X \to X$ be a minimal subshift of $K^{\Z}$.
Suppose $Y \subset X$ has nonempty interior,
and that $Y$ is the intersection of $X$ with
a product subset of $K^{\Z}$ depending only on coordinates
in the interval $n_1, \, n_1 + 1, \, \ldots, \, n_2$ in~$\Z$.
We construct a recursive subhomogeneous algebra~$Q$
whose base spaces
are closed subspaces of finite products of copies of~$K$,
with the number of factors of $K$ equal to the
matrix size plus the length $n_2 - n_1 + 1$ of the interval,
and a \hm{} $\ph \colon Q \to C^* (\Z, X, h)_Y$
such that all elements of
$C^* (\Z, X, h)_Y$ whose coefficients only
depend on the coordinates  $n_1, \, n_1 + 1, \, \ldots, \, n_2$
are in the range of~$\ph$.

We begin by defining
what it means for a map of recursive subhomogeneous algebras
to be given by maps
on the components of decompositions.
The setup is very restrictive,
but suffices for our purposes.

\begin{dfn}\label{D_5621_MapRSHA}
As in Definition~\ref{D_5619_Rshd}
and Notation~\ref{D_6102_A2}, let
\begin{align*}
R
& = \bigg[ \cdots \left[ \left[
  C (T_0, M_{r_0}) \oplus_{C (D_1, M_{r_1})} C (T_1, M_{r_1}) \right]
 \oplus_{C (D_2, M_{r_2})} C (T_2, M_{r_2}) \right]
\\
& \hspace*{3em} {\mbox{}}
 \oplus_{C (D_3, M_{r_3})} \cdots \bigg]
            \oplus_{C (D_m, M_{r_m})} C (T_m, M_{r_m})
\end{align*}
and
\begin{align*}
Q
& = \bigg[ \cdots \left[ \left[
  C (Z_0, M_{s_0}) \oplus_{C (E_1, M_{s_1})} C (Z_1, M_{s_1}) \right]
 \oplus_{C (E_2, M_{s_2})} C (Z_2, M_{s_2}) \right]
\\
& \hspace*{3em} {\mbox{}}
 \oplus_{C (E_3, M_{s_3})} \cdots \bigg]
            \oplus_{C (E_m, M_{s_m})} C (Z_m, M_{s_m})
\end{align*}
be two recursive subhomogeneous algebras
with given recursive subhomogeneous decompositions
of the same length~$m$,
with compact Hausdorff spaces $T_l$ and $Z_l$,
matrix sizes $r_l$ and $s_l$,
and $l$-th stage algebras $R_l$ and $Q_l$
for $l = 0, 1, \ldots, m$
(so that
\[
R_0 = C (T_0, M_{r_0}),
\,\,\,\,\,\,
R_1 = C (T_0, M_{r_0}) \oplus_{C (D_1, M_{r_1})} C (T_1, M_{r_1}),
\,\,\,\,\,\,
\ldots,
\,\,\,\,\,\,
R_m = R,
\]
and similarly for $Q_0, Q_1, \ldots, Q_m, Q$),
and with closed subsets
$D_l \subset T_l$ and $E_l \subset Z_l$,
restriction maps
\[
\rh_l \colon C (T_l, M_{r_l}) \to C (D_l, M_{r_l})
\andeqn
\et_l \colon C (Z_l, M_{s_l}) \to C (E_l, M_{s_l}),
\]
and unital \hm{s}
\[
\bt_l \colon R_{l - 1} \to C (D_l, M_{r_l})
\andeqn
\af_l \colon Q_{l - 1} \to C (E_l, M_{r_l}).
\]
A unital \hm{} $\ph \colon Q \to R$
is {\emph{compatible}}
with these decompositions
if for $l = 0, 1, \ldots, m$ (for $\ph_l$)
and $l = 1, 2, \ldots, m$ (for $\dt_l$)
there are unital \hm{s}
\[
\ph_l \colon C (Z_l, M_{s_l}) \to C (T_l, M_{r_l})
\andeqn
\dt_l \colon C (E_l, M_{s_l}) \to C (D_l, M_{r_l})
\]
with the following properties.
Taking
\[
\ph^{(l)} = \bigoplus_{i = 0}^{l} \ph_i
  \colon \bigoplus_{i = 0}^{l} C (Z_i, M_{s_i})
      \longrightarrow \bigoplus_{i = 0}^{l} C (T_i, M_{r_i}),
\]
we require that
$\ph^{(l)} (Q_l) \subset R_l$,
that the diagrams
\[
\newlength{\hda}
\settoheight{\hda}{$\ph^{(l - 1)} |_{Q_{l - 1}}$}
\begin{CD}
C (Z_l, M_{s_l}) @>{\ph_l\rule{0em}{\hda}}>> C (T_l, M_{r_l})  \\
@VV{\et_l}V  @VV{\rh_l}V            \\
C (E_l, M_{s_l}) @>{\dt_l}>> C (D_l, M_{r_l})
\end{CD}
\andeqn
\begin{CD}
Q_{l - 1} @>{\ph^{(l - 1)} |_{Q_{l - 1}}}>> R_{l - 1}  \\
@VV{\af_l}V  @VV{\bt_l}V            \\
C (E_l, M_{s_l}) @>{\dt_l}>> C (D_l, M_{r_l})
\end{CD}
\]
commute,
and that $\ph^{(m)} |_Q = \ph$.
\end{dfn}

\begin{lem}\label{L_5621_ExistH}
Except for~$\ph$,
let the notation be as in Definition~\ref{D_5621_MapRSHA},
and assume that for $l = 1, 2, \ldots, m$
we have $\ph^{(l)} (Q_l) \subset R_l$
and that the diagrams in Definition~\ref{D_5621_MapRSHA} commute.
Then there is a unital \hm{} $\ph \colon Q \to R$
which is compatible with the decompositions
and such that,
with the vertical maps being the obvious inclusions,
the diagram
\[
\begin{CD}
Q @>{\ph}>> R  \\
@VVV  @VVV            \\
\bigoplus_{l = 0}^{m} C (Z_l, M_{s_l})
 @>{\ph^{(m)}}>> \bigoplus_{l = 0}^{m} C (T_l, M_{r_l})
\end{CD}
\]
commutes.
\end{lem}

\begin{proof}
It is immediate to show by induction on~$l$
that $\ph^{(l)}$ defines a
unital \hm{} $Q_l \to R_l$.
\end{proof}

\begin{lem}\label{L_5621_Inj}
In the situation of Lemma~\ref{L_5621_ExistH},
if $\ph_l$ is injective for $l = 0, 1, \ldots, m$
then $\ph$ is injective.
\end{lem}

\begin{proof}
The map
$\ph^{(m)}
 \colon \bigoplus_{l = 0}^{m} C (Z_l, M_{s_l})
 \to \bigoplus_{l = 0}^{m} C (T_l, M_{r_l})$
is injective.
\end{proof}

We now define an approximating system at the level of spaces for a
recursive subhomogeneous decomposition
of the type used earlier.
Presumably one can make such a definition in greater generality.
In particular,
presumably the matrix sizes in the
approximating recursive subhomogeneous decomposition
need not be the same as those in the original.
However, this setup suffices for our purposes.

Keep in mind that
the set $T_{l, \mu}$ is contained in both $T_{\mu (1)}$ and $T_l$.
In the definition,
we are in several places treating it as being in $T_{\mu (1)}$.
(For example,
$T_{l, \mu, 1} = T_{l, \mu}$
and $h_{l, \mu, 1} = \id_{T_{l, \mu}}$.)

\begin{dfn}\label{D_5621_AppSys}
Let the notation be as in Convention~\ref{N_5618_SuppressY}.
(See the summary in Remark~\ref{R_5618_Summ}.)
For $l = 1, 2, \ldots, m$,
$\mu \in \cS_l$ with length~$t$,
and $s = 1, 2, \ldots, t$,
further define
$T_{l, \mu, s} \subset T_{\mu (s)}$ by
\[
T_{l, \mu, s}
 = h^{- (r_{\mu (1)} + r_{\mu (2)} + \cdots + r_{\mu (s - 1)})}
    (T_{l, \mu})
\]
and
\[
h_{l, \mu, s} = h^{r_{\mu (1)} + r_{\mu (2)} + \cdots + r_{\mu (s - 1)}}
 \colon T_{l, \mu} \to T_{l, \mu, s}.
\]
An
{\emph{approximating system}}
for~$\cY$
consists of \cms{s} $Z_0, Z_1, \ldots, Z_m$,
surjective maps $q_l \colon T_l \to Z_l$
for $l = 0, 1, \ldots, m$,
closed subsets $Z_{l, \mu} \subset Z_l$
such that $q_l^{-1} (Z_{l, \mu}) = T_{l, \mu}$
for $l = 1, 2, \ldots, m$ and $\mu \in \cS_l$,
and closed subsets $Z_{l, \mu, s} \subset Z_{\mu (s)}$
and continuous maps
$g_{l, \mu, s} \colon Z_{l, \mu} \to Z_{l, \mu, s}$
such that the diagram
\begin{equation}\label{Eq_5726_CommDiag}
\begin{CD}
T_{l, \mu}
 @>{h_{l, \mu, s}}>>
 T_{l, \mu, s}  \\
@VV{q_l}V  @VV{q_{\mu (s)}}V            \\
Z_{l, \mu} @>{g_{l, \mu, s}}>> Z_{l, \mu, s}
\end{CD}
\end{equation}
commutes,
for $l = 1, 2, \ldots, m$,
$\mu \in \cS_l$ with length~$t$,
and $s = 1, 2, \ldots, t$.
\end{dfn}

We do not require $q_{\mu (s)}^{-1} (Z_{l, \mu, s}) = T_{l, \mu, s}$.

\begin{lem}\label{L_5621_FromApp}
Let the notation be as in Definition~\ref{D_5621_AppSys},
and let $R_0, R_1, \ldots, R_m$
be as in Definition~\ref{D_5501_S}
(suppressing~$\cY$ as in
Convention~\ref{N_5618_SuppressY}).
Set $R = R_m$.
Then there are a recursive subhomogeneous algebra
\begin{align*}
Q
& = \bigg[ \cdots \left[ \left[
  C (Z_0, M_{r_0}) \oplus_{C (E_1, M_{r_1})} C (Z_1, M_{r_1}) \right]
 \oplus_{C (E_2, M_{r_2})} C (Z_2, M_{r_2}) \right]
\\
& \hspace*{3em} {\mbox{}}
 \oplus_{C (E_3, M_{r_3})} \cdots \bigg]
            \oplus_{C (E_m, M_{r_m})} C (Z_m, M_{r_m}),
\end{align*}
and a compatible \hm{} $\ph \colon Q \to R$
as in Definition~\ref{D_5621_MapRSHA}
(using the recursive subhomogeneous decomposition of~$R$
in Theorem~\ref{T_5619_RSHA})
in which,
for $l = 0, 1, 2, \ldots, m$
or $l = 1, 2, \ldots, m$ as appropriate
and
using the notation of Definition~\ref{D_5621_MapRSHA},
$E_l = \bigcup_{\mu \in \cS_l} Z_{l, \mu}$,
\[
\ph_l \colon C (Z_l, M_{r_l}) \to C (T_l, M_{r_l})
\andeqn
\dt_l \colon C (E_l, M_{r_l}) \to C (D_l, M_{r_l})
\]
are $f \mapsto f \circ q_l$,
and $\af_l \colon Q_{l - 1} \to C (E_l, M_{r_l})$
is determined by
\begin{align*}
& \af_l (c_0, c_1, \ldots, c_{l - 1}) (z)
\\
& \hspace*{3em} {\mbox{}}
  = \diag \big( ( c_{\mu (1)} \circ g_{l, \mu, 1}) (z), \,
               ( c_{\mu (2)} \circ g_{l, \mu, 2}) (z),
               \, \ldots, \,
               ( c_{\mu (t)} \circ g_{l, \mu, t}) (z) \big)
\end{align*}
for
\[
(c_0, c_1, \ldots, c_{l - 1}) \in Q_{l - 1}
 \subset \bigoplus_{i = 0}^{l - 1} C (Z_i, M_{r_i}),
\]
$\mu \in \cS_l$ with length~$t$,
and $z \in Z_{l, \mu}$.
Moreover,
$\ph$ is injective.
Also,
if
\[
a = (a_0, a_1, \ldots, a_m)
  \in R
  \subset \bigoplus_{l = 0}^m C (T_l, M_{r_l}),
\]
and if $a_l$ is in the range of $\ph_l$ for $l = 0, 1 \ldots, m$,
then $a$ is in the range of~$\ph$.
\end{lem}

\begin{proof}
For convenience, take $D_0 = \E$ and $E_0 = \E$.
Define \ca{s} $Q_l$ for $l = 0, 1, \ldots, m$
inductively as follows.
Set $Q_0 = C (Z_0, M_{r_0})$.
Given $Q_{l - 1}$,
take $Q_l \subset \bigoplus_{i = 0}^{l} C (Z_i, M_{r_i})$
to be the set of all
\[
(c_0, c_1, \ldots, c_{l})
 \in Q_{l - 1} \oplus C (Z_l, M_{r_l})
 \subset \bigoplus_{i = 0}^{l} C (Z_i, M_{r_i})
\]
such that for all $t \in \N$,
all $\mu \in \cS_l$ with length~$t$,
and all $z \in Z_{l, \mu}$
we have
\begin{equation}\label{Eq_6221_Star}
c_l (z) = \diag \big( ( c_{\mu (1)} \circ g_{l, \mu, 1}) (z), \,
               ( c_{\mu (2)} \circ g_{l, \mu, 2}) (z),
               \, \ldots, \,
               ( c_{\mu (t)} \circ g_{l, \mu, t}) (z) \big).
\end{equation}
For $l = 1, 2, \ldots, m$
(and, for $\ph_l$, also for $l = 0$),
define
\[
E_l \subset Z_l,
\,\,\,\,\,\,
\ph_l \colon C (Z_l, M_{r_l}) \to C (T_l, M_{r_l}),
\andeqn
\dt_l \colon C (E_l, M_{r_l}) \to C (D_l, M_{r_l})
\]
as in the statement of the lemma.
Let
\[
\ph^{(l)}
  = \bigoplus_{i = 0}^{l} \ph_i \colon
      \bigoplus_{i = 0}^{l} C (Z_i, M_{r_i})
       \to \bigoplus_{i = 0}^{l} C (T_i, M_{r_i}).
\]
Let $\et_l \colon C (Z_l, M_{r_l}) \to C (E_l, M_{r_l})$
be the restriction map.
Also for $\mu \in \cS_l$ with length~$t$, define
$\ta_{l, \mu} \colon
 \bigoplus_{i = 0}^{l - 1} C (Z_i, M_{r_i}) \to C (Z_{l, \mu}, M_{r_l})$
by
\begin{align*}
& \ta_{l, \mu} (c_0, c_1, \ldots, c_{l - 1}) (z)
\\
& \hspace*{3em} {\mbox{}}
  = \diag \big( ( c_{\mu (1)} \circ g_{l, \mu, 1}) (z), \,
               ( c_{\mu (2)} \circ g_{l, \mu, 2}) (z),
               \, \ldots, \,
               ( c_{\mu (t)} \circ g_{l, \mu, t}) (z) \big)
\end{align*}
for
\[
(c_0, c_1, \ldots, c_{l - 1})
 \in \bigoplus_{i = 0}^{l - 1} C (Z_i, M_{r_i})
\andeqn
z \in Z_{l, \mu}.
\]

Let $l \in \{ 0, 1, \ldots, m \}$.
Since $E_l = \bigcup_{\mu \in \cS_l} Z_{l, \mu}$
(by definition),
$D_l = \bigcup_{\mu \in \cS_l} T_{l, \mu}$
(by Lemma \ref{L_6103_Combined}(\ref{L_6103_Combined_DlUnion})),
and $q_l^{-1} (Z_{l, \mu}) = T_{l, \mu}$
for $\mu \in \cS_l$
(by hypothesis; see Definition~\ref{D_5621_AppSys}),
we get
\begin{equation}\label{Eq_6309_DlEl}
D_l = q_l^{-1} (E_l).
\end{equation}

We claim that the following hold for $l = 0, 1, \ldots, m$.
\begin{enumerate}
\item\label{L_5621_FromApp_Tau}
There is a unique \hm{}
$\af_l \colon Q_{l - 1} \to C (E_l, M_{r_l})$
such that for $c \in Q_{l - 1}$,
$\mu \in \cS_l$,
and $z \in Z_{l, \mu}$,
we have $\af_l (c) (z) = \ta_{l, \mu} (c) (z)$.
(This says that the formula for $\af_l$ in the statement of the lemma
makes sense.)
\item\label{L_5621_FromApp_PBack}
$Q_l = Q_{l - 1}
  \oplus_{C (E_l, M_{r_l}), \, \af_l, \, \et_l} C (Z_l, M_{r_l})$.
\item\label{L_5621_FromApp_Comm}
The diagrams
\[
\newlength{\hdb}
\settoheight{\hdb}{$\ph^{(l - 1)} |_{Q_{l - 1}}$}
\begin{CD}
C (Z_l, M_{r_l}) @>{\ph_l\rule{0em}{\hdb}}>> C (T_l, M_{r_l})  \\
@VV{\et_l}V  @VV{\rh_l}V            \\
C (E_l, M_{r_l}) @>{\dt_l}>> C (D_l, M_{r_l})
\end{CD}
\andeqn
\begin{CD}
Q_{l - 1} @>{\ph^{(l - 1)} |_{Q_{l - 1}}}>> R_{l - 1}  \\
@VV{\af_l}V  @VV{\bt_l}V            \\
C (E_l, M_{r_l}) @>{\dt_l}>> C (D_l, M_{r_l})
\end{CD}
\]
commute.
\item\label{L_5621_FromApp_Range}
$\ph^{(l)} (Q_{l}) \subset R_{l}$.
\end{enumerate}
This claim will imply the conclusion of the lemma,
by taking $\ph = \ph^{(m)} |_Q$
and putting $l = m$ in~(\ref{L_5621_FromApp_Range}).

We prove the claim by induction on~$l$.
For $l = 0$,
(\ref{L_5621_FromApp_Tau}) is vacuous,
(\ref{L_5621_FromApp_PBack}) is the definition of~$Q_0$,
(\ref{L_5621_FromApp_Comm}) is vacuous,
and (\ref{L_5621_FromApp_Range}) is immediate from $\ph^{(0)} = \ph_0$.

We now assume that the claim is known for $l - 1$ and prove it for~$l$.

We first claim that for any $\mu \in \cS_l$,
any $x \in T_{l, \mu}$,
and any~$c \in Q_{l - 1}$,
we have
\begin{equation}\label{Eq_L_5621_FromApp_Agree}
\ta_{l, \mu} (c) ( q_l (x) )
 = \bt_{l, \mu} ( \ph^{(l - 1)} (c) ) (x).
\end{equation}
To prove this claim,
let $t$ be the length of~$\mu$.
Using the formula for $\bt_{l, \mu} (b) (x)$
in Definition~\ref{D_5501_S}
and the formula for $\ta_{l, \mu} (c)$ above,
and recalling that
\[
h_{l, \mu, s}
 = h^{r_{\mu (1)} + r_{\mu (2)} + \cdots + r_{\mu (s - 1)}},
\]
we see that we must prove that
\[
(c_{\mu (s)} \circ q_{\mu (s)} \circ h_{l, \mu, s}) (x)
= (c_{\mu (s)} \circ g_{l, \mu, s} \circ q_l) (x)
\]
for $s = 1, 2, \ldots, t$.
This equation is commutativity of the diagram~(\ref{Eq_5726_CommDiag})
in Definition~\ref{D_5621_AppSys}.
The claim is thus proved.

We prove~(\ref{L_5621_FromApp_Tau}).
We have to show that if $\mu, \nu \in \cS_l$,
$c \in Q_{l - 1}$,
and $z \in Z_{l, \mu} \cap Z_{l, \nu}$,
then $\ta_{l, \mu} (c) (z) = \ta_{l, \nu} (c) (z)$.

Since $q_l$ is surjective,
there is $x \in T_l$ such that $q_l (x) = z$.
Since $q_l^{-1} (Z_{l, \mu}) = T_{l, \mu}$
and $q_l^{-1} (Z_{l, \nu}) = T_{l, \nu}$,
we have $x \in T_{l, \mu} \cap T_{l, \nu}$.

Write $c = (c_0, c_1, \ldots, c_{l - 1})$
with $c_i \in C (Z_i, M_{r_i})$ for $i = 0, 1, \ldots, l - 1$,
and set $b_i = \ph_i (c_i)$.
Thus $\ph^{(l - 1)} (c) = (b_0, b_1, \ldots, b_{l - 1})$.
Call this element~$b$.
By~(\ref{L_5621_FromApp_Range})
for $l - 1$,
we have $b \in R_{l - 1}$.
Therefore Lemma~\ref{L_5619_WellDf}
implies that
$\bt_{l, \mu} (b) (x) = \bt_{l, \nu} (b) (x)$.
Applying~(\ref{Eq_L_5621_FromApp_Agree})
twice, once with $\mu$ as given and once with $\nu$ in place of~$\mu$,
we get
\[
\ta_{l, \mu} (c) (z) = \bt_{l, \mu} (b) (x)
\andeqn
\ta_{l, \mu} (c) (z) = \bt_{l, \mu} (b) (x).
\]
So $\ta_{l, \mu} (c) (z) = \ta_{l, \mu} (c) (z)$,
as desired.
This completes the proof of~(\ref{L_5621_FromApp_Tau}).

The relation~(\ref{L_5621_FromApp_PBack})
follows from
the definition of $Q_l$
and the fact that the formula in~(\ref{L_5621_FromApp_Tau})
defines a \hm{} $\af_l \colon Q_{l - 1} \to C (E_l, M_{r_l})$,
since the formula~(\ref{Eq_6221_Star})
now just says $\af_l (c_0, c_1, \ldots, c_{l - 1}) (z) = c_l (z)$.

We next prove~(\ref{L_5621_FromApp_Comm}).
For the second diagram,
let $x \in D_l$.
Set $z = q_l (x)$.
Then $z \in E_l$ by~(\ref{Eq_6309_DlEl}).
Moreover,
there is $\mu \in \cS_l$ such that $x \in T_{l, \mu}$
and $z \in Z_{l, \mu}$.

Now let $c \in Q_{l - 1}$.
Using the definition of $\bt_l$
(in Lemma~\ref{L_5619_DfnSm}) at the first step,
using (\ref{Eq_L_5621_FromApp_Agree})
at the second step,
using the definition of $\af_{l}$ in~(\ref{L_5621_FromApp_Tau})
at the third step,
and using the definition of $\dt_l$ at the fourth step,
we get
\begin{align*}
\big( \bt_l \circ \ph^{(l - 1)} \big) (c) (x)
& = \big( \bt_{l, \mu} \circ \ph^{(l - 1)} \big) (c) (x)
\\
& = \ta_{l, \mu} (c) (z)
  = \af_l (c) (z)
  = ( \dt_l \circ \af_l ) (c) (x).
\end{align*}
This proves commutativity of the second diagram.
Commutativity of the first diagram is immediate
from the definitions of the maps.

Part~(\ref{L_5621_FromApp_Range})
follows immediately from (\ref{L_5621_FromApp_Comm}),
(\ref{L_5621_FromApp_PBack}),
and the definition of~$R_l$.

It remain to prove the last part.
For $l = 0, 1, \ldots, m$,
the map $\ph_l$ is injective because $q_l$ is surjective.
So $\ph$ is injective by Lemma~\ref{L_5621_Inj}.

Now let
\[
a = (a_0, a_1, \ldots, a_m)
  \in R
  \subset \bigoplus_{l = 0}^m C (T_l, M_{r_l}),
\]
and assume that for $l = 0, 1, \ldots, m$
there is $c_l \in C (Z, M_{r_l})$
such that $\ph_l (c_l) = a_l$.
It suffices to prove that
$(c_0, c_1, \ldots, c_m) \in Q$.
We prove by induction on $l$ that
$(c_0, c_1, \ldots, c_l) \in Q_l$.
For $l = 0$ this is trivial.
So suppose it is known for $l - 1$;
we prove it for~$l$.
Let $z \in E_l$;
we must show that
\[
\af_l (c_0, c_1, \ldots, c_{l - 1}) (z) = \et_l (c_l) (z).
\]

Since $q_l$ is surjective,
there is $x \in T_l$ such that $q_l (x) = z$.
Then $x \in D_l$ by~(\ref{Eq_6309_DlEl}).
Since $a \in R$,
we therefore have
$\bt_l (a_0, a_1, \ldots, a_l) (x) = \rh_l (a_l) (x)$.
Using this at the third step,
commutativity of the right square in~(\ref{L_5621_FromApp_PBack})
at the second step,
and commutativity of the left square in~(\ref{L_5621_FromApp_PBack})
at the fourth step,
we get
\begin{align*}
\af_l (c_0, c_1, \ldots, c_{l - 1}) (z)
& = (\dt_l \circ \af_l) (c_0, c_1, \ldots, c_{l - 1}) (x)
  = \bt_l (a_0, a_1, \ldots, a_l) (x)
\\
& = \rh_l (a_l) (x)
  = (\dt_l \circ \et_l) (c_l) (x)
  = \et_l (c_l) (z),
\end{align*}
as desired.
This completes the induction,
and the proof.
\end{proof}

\begin{lem}\label{L_5612_ExpY_0}
Let $X_0$ be a \chs,
let $X \subset X_0$ be closed,
let $Y \subset X$ be closed,
and let $z \in \sint_{X} (Y)$
(the interior of $Y$ with respect to~$X$).
Then there exists a closed subset $Y_0 \subset X_0$
such that $z \in \sint (Y_0) \cap X$
and $Y_0 \cap X \subset \sint_{X} (Y)$.
\end{lem}

\begin{proof}
Choose an open set $U \subset X_0$ such that $U \cap X_0 = \sint_{X} (Y)$.
Choose an open set $V \subset X_0$
such that $z \in V \subset {\ov{V}} \subset U$.
Set $Y_0 = {\ov{V}}$.
Then $z \in V \cap X \subset \sint (Y_0) \cap X$
and
\[
Y_0 \cap X = {\ov{V}} \cap X \subset U \cap X = \sint_{X} (Y).
\]
This completes the proof.
\end{proof}

We now introduce convenient notation related to
shifts and subshifts.

\begin{dfn}\label{D_5611_Shift}
Let $K$ be a set.
The {\emph{shift}} $h_K \colon K^{\Z} \to K^{\Z}$
is the bijection
given by
$h_K (x)_k = x_{k + 1}$
for $x = (x_k)_{k \in \Z} \in K^{\Z}$
and $k \in \Z$.
\end{dfn}

\begin{rmk}\label{R_6227_If}
If $K$ is a topological space,
then $h_K$ is a \hme.
It is the backwards shift,
satisfying
\[
h_K ( \ldots, x_{-2}, x_{-1}, x_0, x_1, x_2, \ldots)
 = ( \ldots, x_{-1}, x_0, x_1, x_2, x_3, \ldots).
\]
This choice is consistent with~\cite{Lnd}.
(See page 237 there.)
With this choice,
if $h \colon X \to X$ is a bijection,
and $f \colon X \to K$ is any function,
then the map $I_f \colon X \to K^{\Z}$ given by
\[
x \mapsto
 \big( \ldots, \, f (h^{-2} (x)), \, f (h^{-1} (x)), \, f (x),
  \, f (h (x)), \, f (h^2 (x)), \, \ldots \big)
\]
is equivariant from $(X, h)$ to $(K^{\Z}, h_K)$.
If $h$ and $f$ are \ct,
so is $I_f$.
\end{rmk}

\begin{cnv}\label{Cv_5619_Subshift}
Let $K$ be a fixed \cms.
We set $X_0 = K^{\Z}$.
We let $h_0 = h_K \colon X_0 \to X_0$
be the backwards shift,
as in Definition~\ref{D_5611_Shift}.
We let $X \subset X_0$ be an
infinite closed subset such that $h_0 (X) = X$
and which is minimal for~$h_0$,
and we let $h \colon X \to X$
be the restriction of $h_0$ to~$X$.
We define $\af_0 \colon C (X_0) \to C (X_0)$
by $\af_0 (f) = f \circ h_0^{-1}$ for $f \in C (X_0)$,
and we define $\af \colon C (X) \to C (X)$
by $\af (f) = f \circ h^{-1}$ for $f \in C (X)$.

For any subset $I \subset \Z$,
we let $p_{I} \colon X_0 \to K^I$
be the obvious projection map,
which for $x = (x_k)_{k \in \Z} \in X_0$
forgets the coordinates $x_k$ for $k \not\in I$.
For any subsets $I, J \subset \Z$ and any function
$\om \colon I \to J$,
we define $\om^* \colon K^J \to K^I$
by $\om^* (x)_k = x_{\om (k)}$
for $k \in I$ and $x = (x_k)_{k \in J} \in K^J$.
Thus,
if we define $\sm \colon \Z \to \Z$ by $\sm (k) = k + 1$
for $k \in \Z$,
then $h_0 = \sm^*$,
and if for $I \subset \Z$ we let $\io_I \colon I \to \Z$
be the inclusion,
then $p_I = \io_I^*$.

For $I \subset \Z$ and $f \in C (X)$,
we say that $f$ is {\emph{constant in coordinates not in~$I$}}
if whenever $x, y \in X$ satisfy $x_k = y_k$ for all $k \in I$,
then $f (x) = f (y)$.
Equivalently,
there is a \cfn{} $g \colon p_I (X) \to \C$
such that $f = g \circ p_I$.
If $a \in C^* (\Z, X, h)$,
we say that $a$ is {\emph{constant in coordinates not in~$I$}}
if in the formal series
$\sum_{n = - \I}^{\I} a_n u^n$ for~$a$,
the function $a_n \in C (X)$
is constant in coordinates not in~$I$ for all $n \in \Z$.
\end{cnv}

\begin{prp}\label{L_5624_RectSet}
Let the notation be as in Convention~\ref{Cv_5619_Subshift}.
Let $n_1, n_2 \in \Z$ with $n_1 \leq n_2$.
For $k \in [n_1, n_2] \cap \Z$
let $C_k \subset K$ be a closed subset
such that $\sint (C_k) \neq \E$,
and define
\[
C = \prod_{k = n_1}^{n_2} C_k
  \subset K^{[n_1, n_2] \cap \Z}.
\]
Let $Y \subset X$ be the closed subset $Y = p_I^{-1} (C) \cap X$,
and assume that $\sint_X (Y) \neq \E$.
Let $\cY$ be the Rokhlin system
associated to~$Y$ as in Lemma~\ref{L_6224_FullTowers}.
Then there are a recursive subhomogeneous algebra
\begin{align*}
Q
& = \bigg[ \cdots \left[ \left[
  C (Z_0, M_{r_0}) \oplus_{C (E_1, M_{r_1})} C (Z_1, M_{r_1}) \right]
 \oplus_{C (E_2, M_{r_2})} C (Z_2, M_{r_2}) \right]
\\
& \hspace*{3em} {\mbox{}}
 \oplus_{C (E_3, M_{r_3})} \cdots \bigg]
            \oplus_{C (E_m, M_{r_m})} C (Z_m, M_{r_m})
\end{align*}
in which,
for $l = 0, 1, \ldots, m$,
the space $Z_l$ is a closed subset of $K^{[n_1, \, n_2 + r_l] \cap \Z}$,
and an injective unital \hm{}
$\ph \colon Q \to C^* (\Z, X, h)_{\cY}$
such that whenever
$a \in C^* (\Z, X, h)_{\cY}$
is constant in coordinates not in $[n_1, n_2] \cap \Z$,
then there is $b \in Q$ such that $\ph (b) = a$.
\end{prp}

\begin{proof}
To simplify the notation,
we use the usual interval notation for subsets of~$\Z$
instead of subsets of~$\R$.
Thus, $[a, b]$ stands for $[a, b] \cap \Z$,
$[a, b)$ stands for $[a, b) \cap \Z$,
etc.
We simplify the notation for $\cY$ and objects associated with it
as in Convention~\ref{N_5618_SuppressY}.
(See the summary in Remark~\ref{R_5618_Summ}.)
Finally,
for $l = 0, 1, \ldots, m$,
for $\mu \in \cS_l$ with length~$t$,
and for $s = 0, 1, 2, \ldots, t$,
we define
\begin{equation}\label{Eq_6224_plmus}
p_{l, \mu, s} = r_{\mu (1)} + r_{\mu (2)} + \cdots + r_{\mu (s)},
\end{equation}
so that
\begin{equation}\label{Eq_6224_Tlmu}
T_{l, \mu}
 = T_l \cap \bigcap_{s = 1}^t h^{- p_{l, \mu, s - 1}} (T_{\mu (s)}).
\end{equation}

We will now construct an approximating system for~$\cY$
as in Definition~\ref{D_5621_AppSys}.

For $l = 0, 1, \ldots, m$
define $Z_l = p_{[n_1, \, n_2 + r_l]} (T_l) \subset K^{[n_1, \, n_2 + r_l]}$,
and set
\[
q_l = p_{[n_1, \, n_2 + r_l]} |_{T_l}
  \colon T_l \to Z_l.
\]
For $\mu \in \cS_l$,
set $Z_{l, \mu} = q_l (T_{l, \mu}) \subset Z_l$.

By definition
(see Lemma~\ref{L_6224_FullTowers}),
\begin{equation}\label{Eq_6224_Tl}
T_l = Y \cap h^{-r_l} (Y),
\end{equation}
so
\begin{align*}
& T_l =
\big\{ x \in X
 \colon {\mbox{$x_k \in C_k$
    for $k = n_1, \, n_1 + 1, \, \ldots, \, n_2$}}
\\
& \hspace*{5em} {\mbox{}}
    {\mbox{and $x_{k} \in C_{k - r_l}$
    for $k = n_1 + r_l, \, n_1 + r_l + 1, \, \ldots, \, n_2 + r_l$}}
     \big\}.
\end{align*}
Therefore
$Z_l \subset K^{[n_1, \, n_2 + r_l]}$
is given by
\begin{align}\label{Eq_6224_IdZl}
& Z_l =
\big\{ x \in p_{[n_1, \, n_2 + r_l]} (X)
 \colon {\mbox{$x_k \in C_k$
    for $k = n_1, \, n_1 + 1, \, \ldots, \, n_2$}}
\\
& \hspace*{5em} {\mbox{}}
    {\mbox{and $x_{k} \in C_{k - r_l}$
    for $k = n_1 + r_l, \, n_1 + r_l + 1, \, \ldots, \, n_2 + r_l$}}
     \big\}.
\notag
\end{align}
It is easily seen from these formulas that
$q_l^{-1} (Z_l) = T_l$.

Let $\mu \in \cS_l$
and let $t$ be the length of~$\mu$.
Combining (\ref{Eq_6224_Tl}), (\ref{Eq_6224_Tlmu}),
(\ref{Eq_6224_plmus}),
and $r_{\mu (1)} + r_{\mu (2)} + \cdots + r_{\mu (t)} = r_l$,
we get
\begin{align*}
T_{l, \mu}
& = Y \cap h^{-r_l} (Y) \cap
   \bigcap_{s = 1}^t h^{- p_{l, \mu, s - 1}}
         \big( Y \cap h^{- r_{\mu (s)}} (Y) \big)
\\
& = h^{- p_{l, \mu, 0}} (Y) \cap h^{- p_{l, \mu, 1}} (Y) \cap
    \cdots \cap h^{- p_{l, \mu, t}} (Y).
\end{align*}
Substituting the definition of~$Y$ in this formula,
we get
\begin{align*}
& T_{l, \mu} =
\big\{ x \in X
 \colon {\mbox{$x_{k} \in C_{k - p_{l, \mu, s}}$
    for $s = 0, 1, \ldots, t$}}
\\
& \hspace*{5em} {\mbox{}}
    {\mbox{and
    $k = n_1 + p_{l, \mu, s}, \, n_1 + p_{l, \mu, s} + 1,
     \, \ldots, \, n_2 + p_{l, \mu, s}$}}
     \big\}.
\end{align*}
Since for $s = 0, 1, \ldots, t$
we have
\[
0 = p_{l, \mu, 0}
  \leq p_{l, \mu, s}
  \leq p_{l, \mu, t}
  = r_l,
\]
all coordinates used in this formula
are in $[n_1, \, n_2 + r_l]$.
Therefore
\begin{align}\label{Eq_6224_IdZlmu}
& Z_{l, \mu} =
\big\{ x \in p_{[n_1, \, n_2 + r_l]} (X)
 \colon {\mbox{$x_{k} \in C_{k - p_{l, \mu, s}}$
    for $s = 0, 1, \ldots, t$}}
\\
& \hspace*{5em} {\mbox{}}
    {\mbox{and
    $k = n_1 + p_{l, \mu, s}, \, n_1 + p_{l, \mu, s} + 1,
     \, \ldots, \, n_2 + p_{l, \mu, s}$}}
     \big\}.
\notag
\end{align}
It is now easily seen that
$q_l^{-1} (Z_{l, \mu}) = T_{l, \mu}$.

For $s = 1, 2, \ldots, t$,
define
\[
\sm_{l, \mu, s} \colon [n_1, \, n_2 + r_{\mu (s)}]
   \to [n_1, \, n_2 + r_l]
\]
by
$\sm_{l, \mu, s} (k) = k + p_{l, \mu, s - 1}$
for $k \in [n_1, \, n_2 + r_{\mu (s)}]$.
Then define
$g_{l, \mu, s}^{(0)} \colon
  Z_{l, \mu} \to K^{[n_1, \, n_2 + r_{\mu (s)}]}$
by $g_{l, \mu, s}^{(0)} = (\sm_{l, \mu, s})^* |_{Z_{l, \mu}}$.
Set
$Z_{l, \mu, s} = g_{l, \mu, s}^{(0)} (Z_{l, \mu}) \subset Z_{\mu (s)}$,
and let $g_{l, \mu, s} \colon Z_{l, \mu} \to Z_{l, \mu, s}$
be the corestriction of $g_{l, \mu, s}^{(0)}$.

We claim that $Z_{l, \mu, s} \subset Z_{\mu (s)}$.
Let $x \in Z_{l, \mu}$;
we show $(\sm_{l, \mu, s})^* (x) \in Z_{\mu (s)}$,
by using~(\ref{Eq_6224_IdZlmu}).
First, let $k \in [n_1, n_2]$.
Then
\[
(\sm_{l, \mu, s})^* (x)_k
 = x_{k + p_{l, \mu, s - 1}},
\]
which is in $C_k$ by using $s - 1$ in place of~$s$
in~(\ref{Eq_6224_IdZlmu}).
Next,
let $k \in [n_1 + r_{\mu (s)}, \, n_2 + r_{\mu (s)}]$.
Set $j = k - r_{\mu (s)}$,
so that $j \in [n_1, n_2]$
and $k + p_{l, \mu, s - 1} = j + p_{l, \mu, s}$.
Then, using~(\ref{Eq_6224_IdZlmu}) at the second last step,
\[
(\sm_{l, \mu, s})^* (x)_{k}
 = x_{k + p_{l, \mu, s - 1}}
 = x_{j + p_{l, \mu, s}}
 \in C_j
 = C_{k - r_{\mu (s)}}.
\]
The claim is proved.

It is clear from the definitions that the diagram
\[
\begin{CD}
K^{\Z}
 @>{h^{p_{l, \mu, s - 1}}}>>
 K^{\Z}  \\
@VV{p_{[n_1, \, n_2 + r_l]}}V  @VV{p_{[n_1, \, n_2 + r_{\mu (s)}]}}V
            \\
K^{[n_1, \, n_2 + r_l]}
   @>{(\sm_{l, \mu, s})^*}>> K^{[n_1, \, n_2 + r_{\mu (s)}]}
\end{CD}
\]
commutes.
It follows that the diagram~(\ref{Eq_5726_CommDiag})
in Definition~\ref{D_5621_AppSys} commutes.
We have now verified that we have an approximating system
for $\cY$
as in Definition~\ref{D_5621_AppSys}.

Let $R$ be the algebra $R_m$ of Theorem~\ref{T_5619_RSHA},
and let $\gm \colon C^* (\Z, X, h)_{\cY} \to R$
be the isomorphism $\gm_m$ from Theorem~\ref{T_5619_RSHA}.
Let
\begin{align*}
Q
& = \bigg[ \cdots \left[ \left[
  C (Z_0, M_{r_0}) \oplus_{C (E_1, M_{r_1})} C (Z_1, M_{r_1}) \right]
 \oplus_{C (E_2, M_{r_2})} C (Z_2, M_{r_2}) \right]
\\
& \hspace*{3em} {\mbox{}}
 \oplus_{C (E_3, M_{r_3})} \cdots \bigg]
            \oplus_{C (E_m, M_{r_m})} C (Z_m, M_{r_m})
\end{align*}
and $\ph \colon Q \to R$
be obtained by
applying
Lemma~\ref{L_5621_FromApp} to the approximating system
we have just constructed.

Let $f \in C (X)$,
let $n \in \Z$,
and suppose that $f$ is constant on coordinates not in $[n_1, n_2]$
and that $f u^n \in C^* (\Z, X, h)_{\cY}$.
We claim that there is $b \in Q$
such that $\ph (b) = \gm (f u^n)$.
First assume that $n \geq 0$.
Then $f$ vanishes on
$\bigcup_{j = 0}^{n - 1} h^j (Y)$ by Proposition~\ref{P_5429_CharOB}.
Let $l \in \{ 0, 1, \ldots, m \}$.
As in Proposition~\ref{L_5429_FromAY},
$\gm (f u^n)_l$
has possibly nonzero matrix entries
$f \circ h^j |_{T_l}$
only for $j \in \{ n, \, n + 1, \, \ldots, \, r_l - 1 \}$
(no nonzero matrix entries if this set is empty).
Since $f$ is constant on coordinates not in $[n_1, n_2]$,
it follows that
$f \circ h^j$ is constant on coordinates not in
$[ n_1 + j, \, n_2 + j ]$.
This set is always contained in $[n_1, \, n_2 + r_l]$,
so there is $c_l \in C (Z_l, M_{r_l})$ such that
$\gm (f u^n)_l = c_l \circ p_{[n_1, \, n_2 + r_l]}$.
Since this is true for all $l \in \{ 0, 1, \ldots, m \}$,
Lemma~\ref{L_5621_FromApp}
now implies that $\gm (f u^n)$ is in the range of~$\ph$.

Now assume $n < 0$.
Let $l \in \{ 0, 1, \ldots, m \}$.
As in Proposition~\ref{L_5429_FromAY},
$\gm (f u^n)_l$
has possibly nonzero matrix entries
$f \circ h^j |_{T_l}$
only for $j \in \{ 0, 1, \ldots, r_l + n - 1 \}$
(no nonzero matrix entries if this set is empty).
As before,
$f \circ h^j$ is constant on coordinates not in
$[ n_1 + j, \, n_2 + j ]$.
Since $r_l + n - 1 < r_l$ (recalling that $n < 0$),
this set is also
contained in $[n_1, \, n_2 + r_l]$ for all allowed
values of~$j$,
and we continue as before.
The claim is proved.

According to Proposition~7.5 of~\cite{PhLg},
if $a \in C^* (\Z, X, h)_{\cY}$ is arbitrary,
then there are $N \in \N$
and $f_{-N}, \, f_{- N + 1}, \, \ldots, \, f_N \in C (X)$
such that $a = \sum_{n = - N}^N f_n u^n$
and $f_n u^n \in C^* (\Z, X, h)_{\cY}$
for all~$n$.
If $a$ is constant on coordinates not in
$[ n_1, n_2 ]$,
then,
by definition, $f_n$ is constant on coordinates not in $[n_1, n_2]$.
By the claim,
$f_n u^n$ is in the range of~$\ph$.
Since this is true for all~$n$,
it follows that $a$ is in the range of~$\ph$.
\end{proof}

\section{Approximation for subshifts}\label{Sec_Approx}

\indent
In this section,
we take the space $K$ of Section~\ref{Sec_OBSubSf}
to be finite dimensional.
We can then bound the radius of comparison
of the recursive subhomogeneous algebra~$Q$
in Proposition~\ref{L_5624_RectSet}.
Combining this result with an equivariant embedding theorem
of a \mh{} in a suitable shift
(from~\cite{Lnd}),
we prove our main theorem.

\begin{lem}\label{L_5627_FcnAppr}
Let the notation be as in Convention~\ref{Cv_5619_Subshift}.
Let $m \in \N$,
let $f_1, f_2, \ldots, f_m \in C (X)$,
let $B_1, B_2, \ldots, B_m \subset X$ be closed subsets,
and let $\ep > 0$.
Suppose that $f_j |_{B_j} = 0$
for $j = 1, 2, \ldots, m$.
Then there exist a finite subset $I \subset \Z$
and $g_1, g_2, \ldots, g_m \in C (K^I)$
such that $\| (g_j \circ p_I) |_X - f_j \| < \ep$
and $(g_j \circ p_I) |_{B_j} = 0$
for $j = 1, 2, \ldots, m$.
\end{lem}

\begin{proof}
By the Tietze Extension Theorem,
we may assume $f_j$ is defined on all of $X_0$
for $j = 1, 2, \ldots, m$.
For $n \in \N$ set $I_n = \{ - n, \, - n + 1, \, \ldots, n \}$.
Then $K^{\Z} = \varprojlim_n K^{I_n}$,
so there exist $n \in \N$ and
$g_1^{(0)}, g_2^{(0)}, \ldots, g_m^{(0)} \in C (K^{I_n})$
such that
\[
\big\| g_j^{(0)} \circ p_{I_n} - f_j \big\| < \frac{\ep}{5}
\]
for $j = 1, 2, \ldots, m$.

For $j = 1, 2, \ldots, m$
define $C_j = p_{I_n} (B_j)$,
which is a closed subset of $K^{I_n}$.
Then $\big\| g_j^{(0)} |_{C_j} \big\| < \frac{\ep}{5}$.
Choose $b_j \in C (K^{I_n})$
such that $0 \leq b_j \leq 1$,
$b_j$ vanishes on $C_j$,
and $b_j (x) = 1$ for every $x$ in the closed set
\[
\left\{ x \in K^{I_n} \colon
 \big| g_j^{(0)} (x) \big| \geq \frac{2 \ep}{5} \right\}.
\]
Define $g_j = b_j g_j^{(0)}$.
Then $g_j |_{C_j} = 0$,
so $(g_j \circ p_I) |_{B_j} = 0$.
For $x \in K^{I_n}$,
if $\big| g_j^{(0)} (x) \big| \geq \frac{2 \ep}{5}$
then $g_j (x) = g_j^{(0)} (x)$.
If $\big| g_j^{(0)} (x) \big| < \frac{2 \ep}{5}$
then also $| g_j (x) | < \frac{2 \ep}{5}$,
so
$\big| g_j^{(0)} (x) - g_j (x) \big| < \frac{4 \ep}{5}$.
Therefore
$\big\| g_j^{(0)} - g_j \big\| \leq \frac{4 \ep}{5}$,
so that
\[
\| g_j \circ p_I - f_j \|
 \leq \big\| g_j - g_j^{(0)} \big\|
   + \big\| g_j^{(0)} \circ p_{I_n} - f_j \big\|
 < \frac{\ep}{5} + \frac{4 \ep}{5}
 = \ep.
\]
This completes the proof.
\end{proof}

\begin{lem}\label{L_5626_AppByCnst}
Let the notation be as in Convention~\ref{Cv_5619_Subshift},
and let $z \in X$.
Then for every finite set $F \subset C^* (\Z, X, h)_{ \{ z \} }$
and every $\ep > 0$
there are a finite set $I \subset \Z$
and closed sets $E_k \subset K$ for $k \in I$
such that,
with $E = \prod_{k \in I} E_k$
and $Y = p_I^{-1} (E) \cap X$,
we have $p_I (z) \in \sint (E)$
and for every $a \in F$
there exists $b \in C^* (\Z, X, h)_Y$
which is constant in coordinates not in~$I$
and satisfies $\| b - a \| < \ep$.
\end{lem}

\begin{proof}
Choose a compact neighborhood $S$ of $z$ in~$X$
and a finite set $F_0 \subset C^* (\Z, X, h)_S$
such that for every $a \in F$ there is $c \in F_0$
with $\| c - a \| < \frac{\ep}{2}$.
Write $F_0 = \{ c_1, c_2, \ldots, c_m \}$.
Choose a finite set $I_0 \subset \Z$
and a compact neighborhood $Y_0$ of $p_{I_0} (z)$ in~$p_{I_0} (X)$
such that $p_{I_0}^{-1} (Y_0) \cap X \subset S$.
Choose compact sets $E_k \subset K$ for $k \in I_0$
such that $z_k \in \sint (E_k)$
and such that the set $Y_1 = \prod_{k \in I_0} E_k$
satisfies $Y_1 \subset Y_0$.
Define $Y = p_{I_0}^{-1} (Y_1) \cap X$,
which is contained in~$S$.
Then $F_0 \subset C^* (\Z, X, h)_Y$.
We have $\sint_X (Y) \neq \E$
since $p_{I_0} (y) \in \sint (Y_1)$.
By Lemma~\ref{L_5429_RTime_bdd}, there is $N \in \N$
such that $\bigcup_{n = 0}^N h^n (Y) = X$.
For $n \in \Z$, set
\[
B_n = \begin{cases}
   \bigcup_{j = 0}^{n - 1} h^j (Y)    & \hspace{3em} n > 0
        \\
   \varnothing                        & \hspace{3em} n = 0
       \\
   \bigcup_{j = 1}^{- n} h^{-j} (Y)   & \hspace{3em} n < 0.
\end{cases}
\]
By Proposition~7.5 of~\cite{PhLg},
there exist $f_{j, n} \in C_0 (X \SM B_n)$
for $j = 1, 2, \ldots, m$
and $n = - N, \, - N + 1, \, \ldots, N$
such that
$c_j = \sum_{n = - N}^N f_{j, n} u^n$
for $j = 1, 2, \ldots, m$.
Use Lemma~\ref{L_5627_FcnAppr}
to find a finite set $I \subset \Z$
and functions $g_{j, n} \in C (K^I)$
such that
\[
\| (g_{j, n} \circ p_I) |_X - f_{j, n} \| < \frac{\ep}{2 (2 N + 1)}
\]
and $(g_{j, n} \circ p_I) |_{B_n} = 0$
for $j = 1, 2, \ldots, m$
and $n = - N, \, - N + 1, \, \ldots, N$.
\Wolog{} $I_0 \subset I$.
Take $E_k = K$ for $k \in I \SM I_0$,
and define $E = \prod_{k \in I} E_k \subset K^I$.
(Up to a coordinate shuffle,
$E = Y_1 \times K^{I \SM I_0}$.)

Let $a \in F$.
Choose $j \in \{ 1, 2, \ldots, m \}$
such that $\| a - c_j \| < \frac{\ep}{2}$.
Set
\[
b = \sum_{n = - N}^N (g_{j, n} \circ p_I) |_X u^n.
\]
Then
\[
\| a - b \|
 \leq \| a - c_j \|
   + \sum_{n = - N}^N \| (g_{j, n} \circ p_I) |_X - f_{j, n} \|
 < \frac{\ep}{2} + \sum_{n = - N}^N \frac{\ep}{2 (2 N + 1)}
 = \ep.
\]
Clearly $b$ is constant in coordinates not in~$I$.
Also,
since $(g_{j, n} \circ p_I) |_{B_n} = 0$
for $n = - N, \, - N + 1, \, \ldots, N$,
Proposition~7.5 of~\cite{PhLg}
shows that $b \in C^* (\Z, X, h)_Y$.
\end{proof}

\begin{lem}\label{L_5626_PvEst}
Let the notation be as in Convention~\ref{Cv_5619_Subshift},
and let $z \in X$.
Assume that the coordinates $z_k$ are all distinct for $k \in \Z$.
Then for every finite set $F \subset C^* (\Z, X, h)_{ \{ z \} }$
and every $\ep > 0$
there are a recursive subhomogeneous algebra
$Q$ such that $\rc (Q) \leq \dim (K)$
and an injective unital \hm{}
$\ph \colon Q \to C^* (\Z, X, h)_{ \{ z \} }$
such that
for every $a \in F$
there exists $b \in Q$
with $\| \ph (b) - a \| < \ep$.
\end{lem}

\begin{proof}
Choose a finite set $I \subset \Z$
and a closed subset $E = \prod_{k \in I} E_k \subset K^{I}$
following Lemma~\ref{L_5626_AppByCnst}.
We may replace $I$ by any larger finite set,
replacing $E$ by its product
with the sets $E_k = K$ for the additional indices~$k$.
This modification does not change the set $p_I^{-1} (E) \cap X$
in Lemma~\ref{L_5626_AppByCnst};
the other part of the conclusion of Lemma~\ref{L_5626_AppByCnst}
(about being constant in coordinates not in~$I$)
is weaker when the set $I$ is larger.
So the conclusion of Lemma~\ref{L_5626_AppByCnst}
still holds after such a replacement.
We can therefore take $I$ to be of the form
\[
I = \{ n_1, \, n_1 + 1, \, \ldots, \, n_2 \}
\]
with $n_1 \leq n_2$.

Choose a metric $\rh$ on~$K$,
and for $y \in K$ and $\dt > 0$
let $B_{\dt} (y)$ be the open ball of radius $\dt$ and center~$y$.
Choose $\dt > 0$
such that for
\[
j, k \in \{ n_1, \, n_1 + 1, \, \ldots, \, n_2 \}
\andeqn
j \neq k,
\]
we have $\rh (z_j, z_k) > 3 \dt$.
Define
\[
C = E \cap \prod_{k = n_1}^{n_2} {\ov{B_{\dt} (z_k)}}
  = \prod_{k = n_1}^{n_2} \big( {\ov{B_{\dt} (z_k)}} \cap E_k \big).
\]
The second expression shows that $C$ is a product set.
Set $Y_0 = p_I^{-1} (E) \cap X$
and $Y = p_I^{-1} (C) \cap X$.
We have $z \in \sint_X (p_I^{-1} (E) \cap X)$
by assumption,
and
\[
p_I^{-1} \left(  \prod_{k = n_1}^{n_2}B_{\dt} (z_k) \right)
\]
is an open subset of $X_0$ which contains~$z$,
so $z \in \sint_X (Y)$.

The sets $I$ and $Y$ satisfy the hypotheses of
Proposition~\ref{L_5624_RectSet}.
Therefore we get a system $\cY$ of Rokhlin towers
as there,
together with a recursive subhomogeneous algebra $Q$
and an injective unital \hm{}
$\ph \colon Q \to C^* (\Z, X, h)_{\cY}$.
The range of $\ph$ is contained in
$C^* (\Z, X, h)_{ \{ z \} }$
because $z \in Y$.
Now let $a \in F$.
The choice of $E$ using Lemma~\ref{L_5626_AppByCnst}
ensures that there is $c \in C^* (\Z, X, h)_{Y_0}$
such that $c$ is constant in coordinates not in~$I$
and $\| c - a \| < \ep$.
Since $Y \subset Y_0$,
we have $c \in C^* (\Z, X, h)_{\cY}$.
It follows from Proposition~\ref{L_5624_RectSet}
that there exists $b \in Q$
such that $\ph (b) = c$.
Thus $\| \ph (b) - a \| < \ep$.

It remains to estimate $\rc (Q)$.

We claim that $h^k (Y) \cap Y = \E$
for $k = 1, 2, \ldots, n_2 - n_1$.
If the claim is false,
there are
\[
k \in \{ 1, 2, \ldots, n_2 - n_1 \}
\andeqn
x \in Y
\]
such that $h^k (x) \in Y$.
We have $n_1 + k \in I$, so
$\rh (x_{n_1 + k}, \, z_{n_1 + k} ) \leq \dt$.
Also from $h^k (x) \in Y$
we get
\[
\rh (x_{n_1 + k}, \, z_{n_1} )
 = \rh ( h^k (x)_{n_1}, \, z_{n_1} )
 \leq \dt.
\]
Therefore
\[
3 \dt
 < \rh ( z_{n_1 + k}, \, z_{n_1} )
 \leq \rh (z_{n_1 + k}, \, x_{n_1 + k} )
           + \rh (x_{n_1 + k}, \, z_{n_1} )
 \leq 2 \dt.
\]
This contradiction proves the claim.

The claim implies that $r_l \geq n_2 - n_1 + 1$
for $l = 0, 1, \ldots, m$.
Since $Q$ has base spaces $Z_l$ and matrix sizes $r_l$
for $l = 0, 1, \ldots, m$,
it follows from Theorem~5.1 of~\cite{Tm9}
that
\begin{equation}\label{Eq_5627_rcRSHA}
\rc (Q)
 \leq \max_{0 \leq l \leq m} \frac{\dim (Z_l) - 1}{2 r_l}.
\end{equation}
We recall (Proposition 3.1.5 of~\cite{Pr})
that if $M$ is a topological space
and $M_0 \subset M$ is closed,
then $\dim (M_0) \leq \dim (M)$.
Also,
by Proposition 3.2.6 of~\cite{Pr}
(``bicompact'' there means compact Hausdorff;
see Definition 1.5.4 of~\cite{Pr}),
if $M_1$ and $M_2$ are compact
then $\dim (M_1 \times M_2) \leq \dim (M_1) + \dim (M_2)$.
For $l = 0, 1, \ldots, m$,
using these facts and
$Z_l \subset K^{[n_1, \, n_2 + r_l] \cap \Z}$
(from Proposition~\ref{L_5624_RectSet})
at the first step,
we get
\[
\dim (Z_l)
 \leq (n_2 - n_1 + 1 + r_l) \dim (K)
 \leq 2 r_l \dim (K).
\]
Substituting this inequality in~(\ref{Eq_5627_rcRSHA})
gives $\rc (Q) \leq \dim (K)$,
as desired.
\end{proof}

We strengthen Lemma~6.1 of~\cite{Niu}.
We first
state for convenience a combination of results from~\cite{BRTTW}.

\begin{prp}[\cite{BRTTW}]\label{P_5620_BRTTW}
Let $A$ be a unital \ca{} all of whose quotients are stably finite.
Then $\rc (A)$ is the infimum of all numbers $r \geq 0$
such that whenever $m, n \in \Nz$ with $n > 0$
and $\et, \mu \in \Cu (A)$ satisfy
\[
\frac{m}{n} > r
\andeqn
(n + 1) \et + m \langle 1_A \rangle_A \leq n \mu,
\]
then $\et \leq \mu$.
\end{prp}

\begin{proof}
In~\cite{BRTTW},
combine Proposition 3.2.1 (with $e = 1_A$),
Definition 3.2.2 and the discussion afterwards,
and Proposition 3.2.3.
\end{proof}

\begin{lem}\label{L_5620_GenNiu}
Let $A$ be a unital \ca{} all of whose quotients are stably finite,
and let $r \in [0, \I)$.
Suppose that for every finite set $F \subset A$ and every $\ep > 0$
there is a unital \ca{} $C$ all of whose quotients are stably finite
and an injective unital \hm{} $\ph \colon C \to A$
such that $\rc (C) < r + \ep$ and for all $a \in F$ we have
$\dist (a, \ph (C)) < \ep$.
Then $\rc (A) \leq r$.
\end{lem}

\begin{proof}
We begin with several preliminaries.
First,
it is easy to see that
the hypotheses imply that
for every finite set $F_0 \subset K \otimes A$,
every finite set $F_1 \subset (K \otimes A)_{+}$,
and every $\ep > 0$
there is a unital \ca{} $C$ all of whose quotients are stably finite
and an injective unital \hm{} $\ph \colon C \to A$
such that $\rc (C) < r + \ep$,
for all $a \in F_0$ we have
$\dist \big( a, \, (\id_K \otimes \ph) (C) \big) < \ep$,
and for all $a \in F_1$ we have
$\dist \big( a, \, (\id_K \otimes \ph) (C_{+}) \big) < \ep$.
Second,
we will take the algebra $C$ to be a subalgebra of~$A$.
Third,
with the help of a fixed isomorphism $K \otimes K \cong K$
and the obvious injective \hm{s}
(with $n$ summands in the first algebra)
\begin{equation}\label{Eq_6309_sumInc}
K \oplus K \oplus \cdots \oplus K
 \longrightarrow M_n \otimes K
 \longrightarrow K \otimes K,
\end{equation}
we can assume that direct sums of elements in $K \otimes A$
are defined in a standard way,
so that,
for example,
\[
\| a_1 \oplus a_2 - b_1 \oplus b_2 \|
 = \max ( \| a_1 - b_1 \|, \, \| a_2 - b_2 \| ),
\]
and similarly for more summands.
Finally,
since we use Cuntz comparison in several algebras at the same time,
we always label the notation with the relevant algebra.

By Proposition~\ref{P_5620_BRTTW},
it suffices to let $m, n \in \Nz$ satisfy $n > 0$
and $\frac{m}{n} > r$,
let $a, b \in (K \otimes A)_{+}$ satisfy
$(n + 1) \langle a \rangle_A + m \langle 1_A \rangle
  \leq n \langle b \rangle_A$
in $\Cu (A)$,
and show that $a \precsim_A b$.
By Proposition~2.6 of \cite{KR},
it suffices to show that
$(a - \ep)_{+} \precsim_A b$ for all $\ep > 0$.

So let $m, n, a, b$ be as above, and let $\ep > 0$.
\Wolog{} $\| a \| \leq 1$ and $\| b \| \leq 1$.
Also, we may assume $\ep < 1$.
Let $x \in (K \otimes A)_{+}$
be the direct sum of $n + 1$ copies of $a$,
let $y \in (K \otimes A)_{+}$
be the direct sum of $n$ copies of $b$,
and let $q \in (K \otimes A)_{+}$
be the direct sum of $m$ copies of~$1_A$.
The relation
$(n + 1) \langle a \rangle_A + m \langle 1 \rangle_A
 \leq n \langle b \rangle_A$
means that
$x \oplus q \precsim_A y$.
By Proposition~2.6 of \cite{KR},
there exists $\dt > 0$ such that
\[
\big( (x \oplus q) - \tfrac{1}{3} \ep \big)_{+}
   \precsim_A (y - \dt)_{+}.
\]
Since $\ep < 3$, this is equivalent to
\[
\big( x - \tfrac{1}{3} \ep \big)_{+} \oplus q \precsim_A (y - \dt)_{+}.
\]
Therefore there is $v \in K \otimes A$ such that
\begin{equation}\label{Eq_6310_xqy}
\big\| \big( x - \tfrac{1}{3} \ep \big)_{+} \oplus q
   - v^* (y - \dt)_{+} v \big\|
 < \frac{\ep}{12}.
\end{equation}
A polynomial approximation argument provides
$\rh > 0$ with
\[
\rh
 \leq \min \left( 1, \, \dt, \, \frac{\ep}{24 ( \| v \| + 1 )} \right),
\]
and so small
that whenever $D$ is a \ca{}
and $z \in D_{+}$ satisifies $\| z \| \leq 1$,
then for all $c \in D_{+}$
such that $\| c - z \| < \rh$,
we have
\[
\big\| (c - \dt )_{+} - (z - \dt )_{+} \big\|
 < \frac{\ep}{24 ( \| v \| + 1 )^2}
\andeqn
\big\| \big( c - \tfrac{1}{3} \ep \big)_{+}
       - \big( z - \tfrac{1}{3} \ep \big)_{+}\big \|
   < \frac{\ep}{12}.
\]

By hypothesis,
there are a unital subalgebra $C \subset A$
all of whose quotients are stably finite
and such that $\rc (C) < \frac{m}{n}$,
and $a_0, b_0 \in (K \otimes C)_{+}$
and $v_0 \in K \otimes C$,
such that
\begin{equation}\label{Eq_6310_abv}
\| a_0 - a \| < \rh,
\,\,\,\,\,\,
\| b_0 - b \| < \rh,
\andeqn
\| v_0 - v \| < \rh.
\end{equation}
Since $C$ is unital,
we have $q \in K \otimes C$.
Let $x_0 \in (K \otimes C)_{+}$
be the direct sum of $n + 1$ copies of $a_0$,
and let $y_0 \in (K \otimes C)_{+}$
be the direct sum of $n$ copies of $b_0$.
Since for each~$n$ we are using a single
map as in~(\ref{Eq_6309_sumInc})
throughout,
we have $\| x_0 - x \| < \rh$
and $\| y_0 - y \| < \rh$.
Also 
$\| v_0 \| \leq \| v \| + 1$.
Then,
using (\ref{Eq_6310_xqy}), (\ref{Eq_6310_abv}),
and the choice of~$\rh$ at the third step,
\begin{align*}
& \big\| \big( x_0 - \tfrac{1}{3} \ep \big)_{+} \oplus q
   - v_0^* (y_0 - \dt)_{+} v_0 \big\|
\\
& \hspace*{3em} {\mbox{}}
\leq \big\| \big( x - \tfrac{1}{3} \ep \big)_{+} \oplus q
   - v^* (y - \dt)_{+} v \big\|
 + \big\| \big( x_0 - \tfrac{1}{3} \ep \big)_{+}
    - \big( x - \tfrac{1}{3} \ep \big)_{+} \big\|
\\
& \hspace*{5em} {\mbox{}}
 + 2 \| v_0 \|^2 \big\| (y_0 - \dt )_{+} - (y - \dt )_{+} \big\|
 + ( \| v_0 \| + \| v \| ) \| (y - \dt)_{+} \| \| v_0 - v \|
\\
& \hspace*{3em} {\mbox{}}
 \leq \big\| \big( x - \tfrac{1}{3} \ep \big)_{+} \oplus q
   - v^* (y - \dt)_{+} v \big\|
 + \big\| \big( a_0 - \tfrac{1}{3} \ep \big)_{+}
    - \big( a - \tfrac{1}{3} \ep \big)_{+} \big\|
\\
& \hspace*{5em} {\mbox{}}
 + 2 ( \| v \| + 1 )^2 \big\| (b_0 - \dt )_{+} - (b - \dt )_{+} \big\|
 + ( 2 \| v \| + 1 ) \| (b - \dt)_{+} \| \| v_0 - v \|
\\
& \hspace*{3em} {\mbox{}}
 < \frac{\ep}{12}
  + \frac{\ep}{12}
  + 2 ( \| v \| + 1 )^2 \left( \frac{\ep}{24 ( \| v \| + 1 )^2} \right)
  + 2 ( \| v \| + 1 ) \| \rh
 \leq \frac{\ep}{3}.
\end{align*}
It follows from Proposition~2.2 of \cite{Rd0}
that
\[
\big[ \big( x_0 - \tfrac{1}{3} \ep \big)_{+} \oplus q
 - \tfrac{1}{3} \ep \big]_{+}
   \precsim_C (y_0 - \dt)_{+}.
\]
Since $\ep < 3$,
this implies
\[
\big( x_0 - \tfrac{2}{3} \ep \big)_{+} \oplus q
   \precsim_C (y_0 - \dt)_{+}.
\]
In $\Cu (C)$ we therefore have
\[
(n + 1) \big\langle \big( a_0 - \tfrac{2}{3} \ep \big)_{+} \big\rangle_C
   + m \langle 1 \rangle_C
 \leq n \langle (b_0 - \dt)_{+} \rangle_C.
\]
Since $\rc (C) < \frac{m}{n}$,
Proposition~\ref{P_5620_BRTTW}
implies that
\[
\big( a_0 - \tfrac{2}{3} \ep \big)_{+} \precsim_C (b_0 - \dt)_{+}.
\]
Using $\| a_0 - a \| < \tfrac{\ep}{3}$
and Corollary~1.6 of~\cite{PhLg}
at the first step,
and $\| b_0 - b \| < \dt$ and Proposition~2.2 of \cite{Rd0}
at the third step,
we get
\[
(a - \ep)_{+}
 \precsim_A \big( a_0 - \tfrac{2}{3} \ep \big)_{+}
 \precsim_C (b_0 - \dt)_{+}
 \precsim_A b,
\]
as was to be proved.
\end{proof}

\begin{thm}\label{T_6227_Subshift}
Let $K$ be a compact metric space,
and let $X \subset K^{\Z}$ be an infinite closed set which is minimal
for the backwards shift on~$K^{\Z}$.
Let $h \colon X \to X$
be the restriction to $X$ of the backwards shift on~$K^{\Z}$.
Suppose that there is $z \in X$
such that the coordinates $z_k$ are all distinct for $k \in \Z$.
Then $\rc ( C^* (\Z, X, h) ) \leq \dim (K)$.
\end{thm}

\begin{proof}
We follow the notation of Convention~\ref{Cv_5619_Subshift}.
Set $A = C^* (\Z, X, h)_{ \{ z \} }$.
We verify that $A$
satisfies the hypotheses of Lemma~\ref{L_5620_GenNiu}
with $r = \dim (K)$.
So let $F \subset A$ be finite and let $\ep > 0$.
Apply Lemma~\ref{L_5626_PvEst}
to get a recursive subhomogeneous algebra
$Q$ such that $\rc (Q) \leq \dim (K)$
and an injective unital \hm{}
$\ph \colon Q \to A$
such that $\dist ( a, \, \ph (Q) ) < \ep$
for all $a \in F$.
All quotients of recursive subhomogeneous algebras
are stably finite,
so the hypotheses of Lemma~\ref{L_5620_GenNiu}
hold.

Lemma~\ref{L_5620_GenNiu} now implies that $\rc (A) \leq \dim (K)$.
Since $A$ is a large subalgebra of~$C^* (\Z, X, h)$
(by Corollary~7.11 of~\cite{PhLg}),
and $C^* (\Z, X, h)$ is infinite dimensional and stably finite,
it follows from Theorem 6.14 of~\cite{PhLg})
that $\rc ( C^* (\Z, X, h) ) = \rc (A)$.
The conclusion follows.
\end{proof}

\begin{thm}\label{T_6227_TrueEst}
Let $X$ be an infinite compact metric space,
and let $h \colon X \to X$ be a \mh.
Let $d \in \N$ satisfy $d > 36 \cdot \mdim (h)$.
Then $\rc ( C^* (\Z, X, h) ) \leq d$.
\end{thm}

\begin{proof}
Set $K = [0, 1]^d$.
Let $C (X, K)$ be the set of all \cfn{s}
from $X$ to~$K$,
which is a closed subset of the \ca{} $C (X) \otimes \C^d$.
For any $f \in C (X, K)$,
we let $I_f \colon X \to K^{\Z}$ be as in Remark~\ref{R_6227_If}.
Theorem~5.1 of~\cite{Lnd}
provides a dense $G_{\dt}$-set $S_0 \subset C (X, K)$
such that for every $f \in S_0$,
the map $I_{f}$ is injective.

Fix $z \in X$.
Since $h$ is minimal and $X$ is infinite,
the points $h^k (z)$,
for $k \in \Z$,
are all distinct.
It is now easy to see that for $m, n \in \Z$ with $m \neq n$,
the set
\[
S_{m, n}
= \big\{ f \in C (X, K) \colon f (h^m (z)) \neq f (h^n (z)) \big\}
\]
is a dense open subset of $C (X, K)$.
By the Baire Category Theorem,
the set
\[
S = S_0 \cap \bigcap_{n \in \Z}
        \bigcap_{m \in \Z \SM \{ n \} } S_{m, n}
\]
is nonempty (in fact, dense).
Choose any $f \in S$.
Then $I_f$ is an equivariant \hme{}
from $X$ to a shift invariant subset of~$K^{\Z}$
such that the coordinates $I_f (z)_k$ are all distinct for $k \in \Z$.
We may therefore assume that $X \subset K^{\Z}$
and that the coordinates $z_k$ are all distinct for $k \in \Z$.
Now Theorem~\ref{T_6227_Subshift}
implies that $\rc ( C^* (\Z, X, h) ) \leq \dim (K) = d$.
\end{proof}

\begin{cor}\label{C_6227_StatedEst}
Let $X$ be a compact metric space,
and let $h \colon X \to X$ be a \mh.
Then $\rc ( C^* (\Z, X, h) ) \leq 1 + 36 \cdot \mdim (h)$.
\end{cor}

\begin{proof}
Let $d$ be the least integer such that $d > 36 \cdot \mdim (h)$.
Then $\rc ( C^* (\Z, X, h) ) \leq d$
by Theorem~\ref{T_6227_TrueEst},
and $d \leq 1 + 36 \cdot \mdim (h)$.
\end{proof}

\section{Open problems}\label{Sec_Prob}

\indent
The obvious open problem is to strengthen the inequality
$\rc ( C^* (\Z, X, h) ) \leq 1 + 36 \cdot \mdim (h)$
to $\rc ( C^* (\Z, X, h) ) \leq \frac{1}{2} \mdim (h)$.
By Theorem~9.3 of~\cite{Gtm},
if $(X, h)$ has a finite dimensional aperiodic factor,
then the constant $36$ in Theorem~5.1 of~\cite{Lnd}
can be improved to~$16$.
In particular,
if $(X, h)$ is minimal,
then $\rc ( C^* (\Z, X, h) ) \leq 1 + 16 \cdot \mdim (h)$.
On the other hand,
by Theorem~1.3 of~\cite{LndTs},
for every $d \in \N$ there exists
a compact metric space~$X$
and a minimal homeomorphism $h \colon X \to X$
such that $\mdim (h) = d / 2$ but $(X, h)$ cannot
be embedded into the the shift on $([0, 1]^d)^{\Z}$.
Thus,
presumably the best possible result gotten by combining the results
in this paper with an embedding theorem for \mh{s}
would be $\rc ( C^* (\Z, X, h) ) \leq 1 + 2\cdot \mdim (h)$.
A proof that
$\rc ( C^* (\Z, X, h) ) \leq \frac{1}{2} \mdim (h)$
seems to require that one work much more carefully
with direct systems of recursive subhomogeneous algebras.
Nevertheless,
there are potentially interesting smaller improvements,
some of which lead to questions which are interesting
in their own right.

Let $h \colon X \to X$
be a minimal homeomorphism.
It is known,
and easy to prove (see Proposition~2.7 of~\cite{LW})
that $\mdim (h^n) = n \cdot \mdim (h)$.
If the conjecture $\rc ( C^* (\Z, X, h) ) = \frac{1}{2} \mdim (h)$
is correct,
then the following conjecture must hold.

\begin{cnj}\label{Cj_6301_hnh}
Let $X$ be a \cms,
let $n \in \N$,
and let $h \colon X \to X$
be a homeomorphism such that $h^n$ is minimal.
Then
\[
\rc ( C^* (\Z, X, h^n) ) = n \cdot \rc ( C^* (\Z, X, h) ).
\]
\end{cnj}

Proving this conjecture directly would enable one to improve
the inequality
$\rc ( C^* (\Z, X, h) ) \leq 1 + 36 \cdot \mdim (h)$
to
\begin{equation}\label{Eq_6301_StSt}
\rc ( C^* (\Z, X, h) ) \leq 36 \cdot \mdim (h),
\end{equation}
and thus deduce the main result of~\cite{EllNiu},
as follows.
Suppose first that there are infinitely many $n \in \N$
such that $h^n$ is minimal.
For such~$n$,
Corollary~\ref{C_6227_StatedEst} gives
$\rc ( C^* (\Z, X, h^n) ) \leq 1 + 36 \cdot \mdim (h^n)$.
Applying Conjecture~\ref{Eq_6301_StSt}
and Proposition~2.7 of~\cite{LW},
and dividing by~$n$,
would give
\[
\rc ( C^* (\Z, X, h) ) \leq \frac{1}{n} + 36 \cdot \mdim (h).
\]
So~(\ref{Eq_6301_StSt}) would follow.
Otherwise,
there must exist infinitely many prime numbers $p$
such that $h^p$ is not minimal.
For such~$p$,
let $Z \subset X$ be a nonempty minimal set for $h^p$.
For any two of the sets
$Z, h (Z), \ldots, h^{p - 1} (Z)$,
the intersection is invariant under~$h^p$,
so these two sets must be equal or disjoint.
Since $p$ is prime,
if any two are equal then they all are,
so the sets on this list are all distinct.
These sets cover~$X$,
so $X$ has at least $p$ connected components.
Since $p$ can be arbitrarily large,
we see that $X$ has infinitely many connected components.
Now $\rc ( C^* (\Z, X, h) ) \leq \frac{1}{2} \cdot \mdim (h)$
by~\cite{HPT}.

To approach Conjecture~\ref{Eq_6301_StSt},
let $u \in C^* (\Z, X, h)$ be the standard unitary,
as in Notation~\ref{N_5429_CPN}.
Let $\bt \colon \Z / n \Z \to \Aut \big( C^* (\Z, X, h) \big)$
be the action obtained by restricting the dual action on
$C^* (\Z, X, h)$ to $\Z / n \Z \subset S^1$.
Thus,
$\bt$ is generated by the automorphism
$\bt_1 \in \Aut \big( C^* (\Z, X, h) \big)$
such that $\bt_1 (f) = f$ for all $f \in C (X)$
and $\bt_1 (u) = e^{2 \pi i / n} u$.
The fixed point algebra $C^* (\Z, X, h)^{\bt}$ is $C^* (\Z, X, h^n)$.
So Conjecture~\ref{Eq_6301_StSt} could be proved by finding
suitable conditions on an action $\bt \colon \Z / n \Z \to \Aut (B)$
under which under which
\begin{equation}\label{Eq_6301_rcFP}
\rc ( B^{\bt} ) = n \cdot \rc (B).
\end{equation}
To get a heuristic idea of what such conditions should be,
consider the following three elementary examples:
\begin{enumerate}
\item\label{6301_hnh_Triv}
The trivial action of $\Z / n \Z$ on~$\C$.
\item\label{6301_hnh_Trans}
The action of $\Z / n \Z$ on $\C^n \cong C (\Z / n \Z)$
coming from the action of $\Z / n \Z$ on itself by translation.
\item\label{6301_hnh_AdReg}
With $\om = e^{2 \pi i / n}$,
the action of $\Z / n \Z$ on $M_n$ generated by
conjugation by
$\diag (1, \om, \ldots, \om^{n - 1} )$.
\end{enumerate}
The relation~(\ref{Eq_6301_rcFP})
holds with $B$ and $\bt$ as in (\ref{6301_hnh_AdReg})
but not as in (\ref{6301_hnh_Triv}) or~(\ref{6301_hnh_Trans}).

It seems perhaps more useful to consider crossed products,
recalling that the original algebra is the fixed point algebra
of the dual action on the crossed product.
We thus want conditions on an action
$\af \colon \Z / n \Z \to \Aut (A)$
which imply that
(\ref{Eq_6301_rcFP}) holds with $\bt = {\widehat{\af}}$,
that is,
\[
\rc \big( C^* ( \Z / n \Z, \, A, \, \af ) \big)
 = \frac{1}{n} \rc (A).
\]
This happens
if $\af$ is the action
in~(\ref{6301_hnh_Trans}) above,
which suggests the following question.

\begin{qst}\label{Q_6301_rcCP}
Let $A$ be a simple unital \ca.
Suppose that $G$ is a finite group and that
$\af \colon G \to \Aut (A)$ is an action satisfying
a sufficiently strong outerness condition,
such as a higher dimensional Rokhlin property
or a version of the tracial Rokhlin property
which doesn't use projections.
Does it follow that $\rc (C^* (G, A, \af)) = \card (G)^{-1} \rc (A)$?
\end{qst}

The easiest case in Question~\ref{Q_6301_rcCP}
is surely when $\af$ has the Rokhlin property.
Unfortunately, this case is likely to be not very interesting.
In many cases (for example, see Theorem~3.5 of~\cite{Iz2}),
if $\af$ has the Rokhlin property
then $A$ is stable under tensoring with a UHF algebra.
If $A$ is simple and stably finite,
then automatically $\rc (A) = 0$.

In a different direction,
we point out that the embedding result we use,
Theorem~5.1 of~\cite{Lnd},
does not require minimality;
it only requires that the system $(X, h)$
have a factor system
which is a \mh{} of an infinite compact metric space.
This suggests that one might try to generalize the inequality
$\rc ( C^* (\Z, X, h) ) \leq 1 + 36 \cdot \mdim (h)$
(Corollary~\ref{C_6227_StatedEst})
or the conjectured inequality
$\rc ( C^* (\Z, X, h) ) \leq \frac{1}{2} \mdim (h)$
to \hme{s} of compact metric spaces which have such factors.
Explicitly,
consider the following question.

\begin{qst}\label{Q_6302_MinFactor}
Let $X$ be a compact metric space,
and let $h \colon X \to X$ be a \hme.
Suppose there are an infinite compact metric space~$Y$,
a \mh{} $k \colon Y \to Y$,
and a \ct{} surjection $g \colon X \to Y$
such that $g \circ h = k \circ g$.
Suppose that $C^* (\Z, X, h)$ is finite.
Does it follow that
$\rc ( C^* (\Z, X, h) ) \leq 1 + 36 \cdot \mdim (h)$?
Does it follow that
$\rc ( C^* (\Z, X, h) ) \leq \frac{1}{2} \mdim (h)$?
\end{qst}

It seems likely that some condition
like finiteness
will be needed.
The radius of comparison is not expected to behave well
without finiteness,
but the other conditions in Question~\ref{Q_6302_MinFactor}
do not imply that $C^* (\Z, X, h)$ is finite.
For example,
let $Y$ be the Cantor set,
let $k \colon Y \to Y$ be any \mh,
take $X = \big( \Z \cup \{ \pm \I \} \big) \times Y$,
take $g \colon X \to Y$ to be the projection to the
second coordinate,
and take $h \colon X \to X$ to be
$h (n, y) = (n + 1, \, g (y) )$
for $y \in Y$ and $n \in \Z \cup \{ \pm \I \}$,
with the convention $- \I + 1 = - \I$
and $\I + 1 = \I$.
Then $C^* (\Z, X, h)$ is not finite.

One major difficulty in answering Question~\ref{Q_6302_MinFactor}
is the generalization of the theory of large subalgebras
to the nonsimple case.
Nothing is known about this.


\begin{thebibliography}{33}

\bibitem{BRTTW}
B.~Blackadar, L.~Robert, A.~P.\  Tikuisis, A.~S.\  Toms, and W.~Winter,
{\emph{An algebraic approach to the radius of comparison}},
Trans.\  Amer.\  Math.\  Soc.\  {\textbf{364}}(2002), 3657--3674.

\bibitem{EHT}
G.~A.\  Elliott, T.~M.\  Ho, and A.~S.\  Toms,
{\emph{A class of simple C*-algebras with stable rank one}},
J.~Funct.\  Anal.\  {\textbf{256}}(2009), 307--322.

\bibitem{EllNiu}
G.~A.\  Elliott and Z.~Niu,
{\emph{The C*-algebra of a minimal homeomorphism of zero mean
  dimension}},
preprint 2014.

\bibitem{GK} J.~Giol and D.~Kerr,
{\emph{Subshifts and perforation}},
J.~reine angew.\  Math.\  {\textbf{639}}(2010), 107--119.

\bibitem{Gtm} Y.~Gutman,
{\emph{Mean dimension and Jaworski-type theorems}},
Proc.\  Lond.\  Math.\  Soc.\  (3) {\textbf{111}}(2015), 831--850.

\bibitem{HPT} T.~Hines, N.~C.\  Phillips, and A.~S.\  Toms,
in preparation.

\bibitem{KR} E.~Kirchberg and M.~R{\o}rdam,
{\emph{Non-simple purely infinite C*-algebras}},
Amer.\  J.\  Math.\  {\textbf{122}}(2000), 637--666.

\bibitem{Iz2} M.~Izumi,
{\emph{Finite group actions on C*-algebras with the
   Rohlin property.~II}},
Adv.\  Math.\  {\textbf{184}}(2004), 119--160.

\bibitem{LP} Q.~Lin and N.~C.\  Phillips,
{\emph{Direct limit decomposition for C*-algebras
of minimal diffeomorphisms}},
pages 107--133 in:
{\emph{Operator Algebras and Applications}}
(Adv.\  Stud.\  Pure Math.\  vol.~38),
Math.\  Soc.\  Japan, Tokyo, 2004.

\bibitem{Lnd} E.~Lindenstrauss, {\emph{Mean dimension,
small entropy factors and an embedding theorem}},
Inst.\  Hautes \'{E}tudes Sci.\  Publ.\  Math.\   %
{\textbf{89}}(2000), 227--262.

\bibitem{LndTs} E.~Lindenstrauss and M.~Tsukamoto,
{\emph{Mean dimension and an embedding problem: an example}},
Israel J.\  Math.\  {\textbf{199}}(2014), 573--584.

\bibitem{LW} E.~Lindenstrauss and B.~Weiss,
{\emph{Mean topological dimension}},
Israel J.\  Math.\  {\textbf{115}}(2000), 1--24.

\bibitem{Niu} Z.~Niu,
{\emph{Mean dimension and AH-algebras with diagonal maps}},
J.~Funct.\  Anal.\  {\textbf{266}}(2014), 4938--4994.

\bibitem{Pr} A.~R.\  Pears,
{\emph{Dimension Theory of General Spaces}},
Cambridge University Press, Cambridge, London, New York, Melbourne,
1975.

\bibitem{Pd} G.~K.\  Pedersen,
{\emph{Pullback and pushout constructions in C*-algebra theory}},
J.~Funct.\  Anal.\  {\textbf{167}}(1999), 243--344.

\bibitem{Ph6} N.~C.\  Phillips,
{\emph{Recursive subhomogeneous algebras}},
Trans.\  Amer.\  Math.\  Soc.\  {\textbf{359}}(2007), 4595--4623.

\bibitem{PhLg}
N.~C.\  Phillips,
{\emph{Large subalgebras}},
preprint (arXiv: 1408.5546v1 [math.OA]).

\bibitem{Pt1} I.~F.\  Putnam,
{\emph{The C*-algebras associated with minimal
homeomorphisms of the Cantor set}},
Pacific J.\  Math.\  {\textbf{136}}(1989), 329--353.

\bibitem{Rd0} M.~R{\o}rdam,
{\emph{On the structure of simple C*-algebras tensored with
   a UHF-algebra. II}},
J.~Funct.\  Anal.\  {\textbf{107}}(1992), 255--269.

\bibitem{Tm1} A.~S.\  Toms,
{\emph{Flat dimension growth for C*-algebras}},
J.~Funct.\  Anal.\  {\textbf{238}}(2006), 678--708.

\bibitem{Tm7} A.~S.\  Toms,
{\emph{An infinite family of non-isomorphic C*-algebras
  with identical K-theory}},
Trans.\  Amer.\  Math.\  Soc.\  {\textbf{360}}(2008), 5343--5354.

\bibitem{Tm9} A.~S.\  Toms,
{\emph{Comparison theory and smooth minimal C*-dynamics}},
Commun.\   Math.\   Phys.\  {\textbf{289}}(2009), 401--433.


\end{thebibliography}
\end{document}